\documentclass{article}
\usepackage[utf8]{inputenc}
\usepackage{mathrsfs}
\usepackage{amsmath}
\usepackage{mathtools}
\usepackage{amsfonts}
\usepackage{tikz-cd}
\usepackage{mathrsfs}
\usepackage{amssymb}
\usepackage{amsthm}
\usepackage{tabularx}
\usepackage{array}
\newtheorem*{theorem*}{Theorem}
\theoremstyle{definition}\newtheorem{definition}{Definition}[section]
\theoremstyle{plain}\newtheorem{theorem}{Theorem}[section]

\newtheorem{lemma}[theorem]{Lemma}

\usepackage{graphicx}
\usepackage{bbm}
\usepackage{MnSymbol,wasysym}
\usepackage{tipa}
\usepackage{MnSymbol,wasysym}
\usepackage{scalerel}
\usepackage{stackengine,wasysym}
\usepackage{blindtext}
\usepackage{caption}
\usepackage{cite}
\usepackage{comment}
\usepackage{indentfirst}
\usepackage{esint}
\usepackage[title]{appendix}
\usepackage{stmaryrd}

\usepackage[margin=1in]{geometry}
\usepackage{geometry}
\newcommand{\limmy}[3]{\underset{#1 \to #2}{\mathrm{lim}}\text{ }#3}

\title{$q$-Deformed Discrete Whittaker Processes}
\author{Richard Ninness}
\date{}

\begin{document}

\maketitle

\begin{abstract}
    We consider a Markov process on non-negative integer arrays of a certain shape, this shape being determined by general parameters in the model which correspond to drifts. In the case where these drifts are trivial, the arrays are reverse plane partitions. This model is closely related to the Whittaker functions of the quantum group $U_q(\mathfrak{sl}_{r+1})$ and the Toda lattice. Moreover, the Markov process we introduce can be viewed as a one-parameter deformation of the discrete Whittaker processes introduced by Neil O'Connell in \cite{neil}. We show that the $q$-deformed Markov process has non-trivial Markov projections to skew shapes.
\end{abstract}

\section{Introduction}
\subsection{The Model} 
\indent Let $\lambda=(\lambda_1,\lambda_2,\dots)$ be a Young diagram, i.e., a weakly decreasing (finite) sequence of nonnegative integers $\lambda_1 \geq \lambda_2 \geq \dots\geq 0$. We will often consider the Young diagram $\delta_{r}=(r-1,r-2,\dots,1)$ for $r \geq 2$, which is referred to as a staircase shape, with $\delta_1=\emptyset$ being the empty Young diagram. A reverse plane partition filling $\lambda$ is an array $\pi=(\pi_{ij})_{(i,j)\in\lambda}$ of nonnegative integers which are weakly increasing along rows and columns.
\begin{figure}[h]
    \centering
    \captionsetup{justification=centering}
    \includegraphics[scale=0.5]{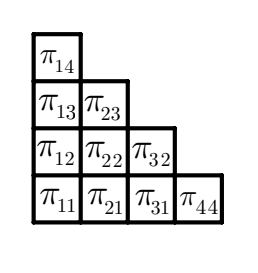}
    \caption{Reverse plane partition with staircase shape $\delta_5=(4,3,2,1)$.}
    \label{fig:fig1}
\end{figure}
We will work with shapes other than Young diagrams, such as skew Young diagrams $\lambda/\mu$ for Young diagrams $\mu \subset \lambda$. The skew shape $\lambda/\mu$ is defined to be the set-theoretic difference $\lambda\backslash \mu$. Naturally, one can restrict reverse plane partitions filling $\lambda$ to subsets $\lambda/\mu$ in order to get reverse plane partitions filling a skew Young diagram. We refer to the restriction from $(\pi_{ij})_{(i,j)\in\lambda}$ to $(\pi_{ij})_{(i,j)\in\lambda/\mu}$ as a projection onto boundary values, where the elements of the latter array are the boundary values. Our goal in this paper is to consider Markov processes on reverse plane partitions filling $\lambda$ whose projection onto certain boundary values is still Markovian. \newline
\indent The simplest case of the main object considered in this paper is a Markov process on reverse plane partitions, say ones of shape $\delta_{r+1}$, such that at each position $(i,j)$ there is an exponential clock of rate $q^{\pi_{i+1,j-1}-\pi_{ij}}(1-q^{\pi_{ij}-\pi_{i,j-1}})(1-q^{\pi_{ij}-\pi_{i-1,j}})$. Once this clock rings, the particle at $(i,j)$ will jump from $\pi_{ij} \in \mathbb{N}$ to $\pi_{ij}-1 \in \mathbb{N}$. We use the conventions that $\pi_{ij}=0$ if $i\leq0$ or $j\leq 0$, and moreover if $\pi_{ij}=0$ when the exponential clock at $(i,j)$ rings then $\pi_{ij}$ will stay at $0$. Note that $\pi=(0)_{(i,j) \in \delta_{r+1}}$ is an absorbing state for this Markov process. We will also incorporate drifts $\alpha_1,\dots,\alpha_r$ through parameters $z_{i,j}=q^{\alpha_{i+1}+\dots+\alpha_{j}}$, and in this more general setting the particle at position $(i,j)$ will jump up by one at a rate $z_{i,r}q^{\pi_{i+1,j-1}-\pi_{ij}}(1-q^{\pi_{ij}-\pi_{i,j-1}})(1-q^{\pi_{ij}-\pi_{i-1,j}}z_{i-1,i+j-1})$. This is a model of a continuous time interacting particle process, with interactions between nearest-neighbors across the two rows $(\pi_{ij})_{i+j=s}$ and $(\pi_{ij})_{i+j=s+1}$ being present for all $s=2,\dots,r$. Moreover, there are slightly weaker interactions between nearest-neighbors in a single row $(\pi_{ij})_{i+j=s}$ for $s=2,\dots,r+1$ which are determined by the $q^{\pi_{i+1,j-1}-\pi_{ij}}$ factor in the jump rates.
\begin{figure}[h]
    \centering
    \captionsetup{justification=centering}
    \includegraphics[scale=0.4]{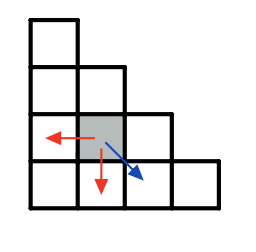}
    \caption{Depicted is a staircase $\delta_{r+1}$, red arrows for nearest-neighbor interactions between adjacent diagonals, blue arrow for nearest-neighbor interaction within an individual diagonal.}
    \label{fig:fig2}
\end{figure}\newline
\indent The above can be generalized to a Markov process on nonnegative integer arrays $(\pi_{ij})_{(i,j) \in \lambda}$ satisfying $\pi_{ij}\geq \pi_{i,j-1},\pi_{i-1,j}-\alpha_i-\dots-\alpha_{i+j-1}$, where $\lambda$ is a fixed Young diagram. Let $\ell(\lambda)$ denote the length of the Young diagram $\lambda$, i.e. $\lambda_i=0$ for $i > \ell(\lambda)$ and $\ell(\lambda)$ is minimal. Denote by $\lambda^{\circ}$ the set of points $(i,j) \in \lambda$ for which $(i,j+1) \in \lambda$ and $(i+1,j) \in \lambda$. We will also fix an additional Young diagram $\mu \subset \lambda^{\circ}$, we should imagine the boundary of $\lambda/\mu$ as dividing $\lambda$ into the two distinct regions $\lambda/\mu$ and $\mu$. When generalizing the above Markov process to a reverse plane partition filling $\lambda$ with boundary values indexed by $\lambda/\mu$, a set of problematic points $\mathrm{Vert}_{\mu}(\lambda)$ arises. Define
\[\mathrm{Vert}_{\mu}(\lambda)=\{(i,j) \in \lambda/\mu : (i-1,j) \in \mu, (i,j-1) \in \lambda/\mu\}\]
to be the set of points whose components are certain vertical paths directed downward along the exterior of $\mu$. (Figure \ref{fig:fig3} depicts the set $\mathrm{Vert}_{\mu}(\lambda)$.) Note that $\mathrm{Vert}_{\mu}(\lambda)=\emptyset$ if $\mu$ is a staircase shape. For points $(i,j) \in \mathrm{Vert}_{\mu}(\lambda)$ we assign the particle at that position an exponential clock of rate $z_{i,\ell(\lambda)}z_{i,i+j-1}^{-1}q^{\pi_{i,j-1}-\pi_{ij}}(1-q^{\pi_{ij}-\pi_{i,j-1}})(1-z_{i-1,i+j-1}q^{\pi_{ij}-\pi_{i-1,j}})$. If $(i,j) \in \lambda \backslash \mathrm{Vert}_{\mu}(\lambda)$ then we assign an exponential clock of rate $z_{i,\ell(\lambda)}q^{\pi_{i+1,j-1}-\pi_{ij}}(1-q^{\pi_{ij}-\pi_{i,j-1}})(1-z_{i-1,i+j-1}q^{\pi_{ij}-\pi_{i-1,j}})$ as usual. One can view $\lambda/\mu$ as a barrier, and when edges cross this barrier they introduce different interaction among nearest-neighbors.
\begin{figure}[h]
    \centering
    \captionsetup{justification=centering}
    \includegraphics[scale=0.4]{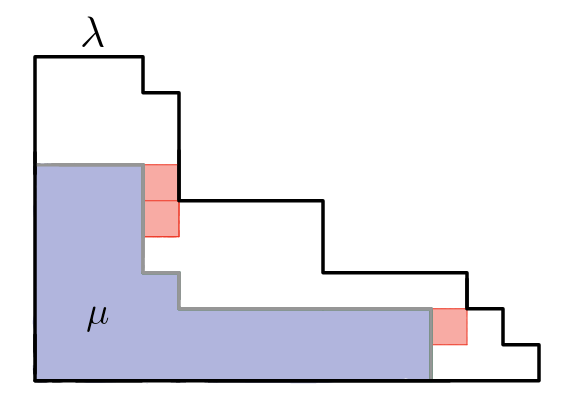}
    \caption{Depicted is a Young diagram $\lambda$ with $\mu \subset \lambda^{\circ}$. The grey border signifies the boundary of $\lambda/\mu$, the red shaded squares correspond to points in $\mathrm{Vert}_{\mu}(\lambda)$.}
    \label{fig:fig3}
\end{figure}\newline
\indent Our model is a $q$-deformation of the one considered in \cite{neil}, with many of the relevant quantities converging---after proper rescaling---to those in \cite{neil} as $q \to 1$. The interactions between nearest-neighbors across rows $(\pi_{ij})_{i+j=s},(\pi_{ij})_{i+j=s+1}$ are the natural $q$-deformation of the interactions present in the model from \cite{neil}, but the interaction among nearest-neighbors coming from the $q^{\pi_{i+1,j-1}-\pi_{ij}}$, $q^{\pi_{i,j-1}-\pi_{ij}}$ factors is unique to this model and will disappear as $q \to 1$.
\subsection{Main Results}
\indent Consider the Markov process on reverse plane partitions of staircase shape $\delta_{r+1}=(r,r-1,\dots,1)$ with generator
\begin{equation}\label{introthingy1}G^r = \sum_{(i,j) \in \delta_{r+1}} q^{\pi_{i+1,j-1}-\pi_{ij}}(1-q^{\pi_{ij}-\pi_{i,j-1}})(1-q^{\pi_{ij}-\pi_{i-1,j}})D_{\pi_{ij}}\end{equation}
where $D_{\pi_{ij}}F=F(\dots,\pi_{ij}-1,\dots)-F(\dots,\pi_{ij},\dots).$ In section $3$ we prove the following
\begin{theorem}\label{smalltheorem}
    Let $\pi(t)$ be a Markov process on $\Pi^{\delta_{r+1}}$. Suppose that, initially, the conditional law of $\pi(t)$ is proportional to
    \[\widetilde{W}_{\delta_{r+1},\delta_r}(\pi;q):=q^{\displaystyle \sum_{(i,j) \in \delta_{r}}\pi_{ij}(\pi_{ij}-\pi_{i+1,j-1})} \prod_{(i,j) \in \delta_r} \binom{\pi_{i+1,j}}{\pi_{ij}}_q \binom{\pi_{i,j+1}}{\pi_{ij}}_q. \]
    Then, if $\pi$ evolves according to $G^{r}$, the boundary values $\pi(t)|_{\delta_{r+1}/\delta_{r}}$ will evolve as a Markov process on $\Pi^{\delta_{r+1}/\delta_r}$ with an explicit generator given by \eqref{doobtransform}.
\end{theorem}
The precise nature of the generator of $\pi(t)|_{\delta_{r+1}/\delta_r}$ is given in \eqref{doobtransform}. In brief though it is a Doob transform of a difference operator with respect to coefficients coming from the eigenfunctions of the quantum difference Toda Hamiltonian
\begin{equation}\label{ok16}
   \mathfrak{H}_q^r = \sum_{i=0}^r \bigg[\mathfrak{D}_{i+1}\mathfrak{D}_i^{-1}(1-y_i)-1\bigg]
\end{equation}
where $\mathfrak{D}_i$ is the operator which multiplies the $y_i$ variable by $q$. The specific difference operator whose Doob transform will be taken is the one coming from the corresponding Toda recursion of the eigenfunction coefficients. \newline
\indent Next, fix the Young diagrams $\lambda$ and $\mu \subset \lambda^{\circ}$ and generalize the generator $G^r$ to
\[\begin{split}
    G^{\lambda,\alpha}_{\mu}&= \sum_{(i,j) \in \lambda \backslash \mathrm{Vert}_{\mu}(\lambda)}z_{i,\ell(\lambda)}q^{\pi_{i+1,j-1}-\pi_{ij}}(1-q^{\pi_{ij}-\pi_{i,j-1}})(1-z_{i-1,i+j-1}q^{\pi_{ij}-\pi_{i-1,j}}) D_{\pi_{ij}} \\
    &+\sum_{(i,j) \in \mathrm{Vert}_{\mu}(\lambda)}z_{i,\ell(\lambda)}z_{i,i+j-1}^{-1}q^{\pi_{i,j-1}-\pi_{ij}}(1-q^{\pi_{ij}-\pi_{i,j-1}})(1-z_{i-1,i+j-1}q^{\pi_{ij}-\pi_{i-1,j}})D_{\pi_{ij}}
\end{split}\]
This generates a Markov process on $\Pi^{\lambda,\alpha}$, which for the purposes of this introduction we refer to as $(\pi^{\lambda}_{\mu}(t))_{t \geq 0}$ due to its dependence on $\lambda$ and $\mu$. \newline
\indent In \cite{neil} the author found a Markov process $\pi^{\lambda}(t)$ on $\Pi^{\lambda,\alpha}$ such that the projection $\pi^{\lambda}(t)|_{\lambda/\mu}$ to every skew shape $\lambda/\mu$, with $\mu \subset \lambda^{\circ}$, was still Markovian. We will prove a weaker result in the $q$-setting, which is that for every skew shape $\lambda/\mu$, with $\mu \subset \lambda^{\circ}$, there exists a Markov process $\pi_{\mu}^{\lambda}(t)$ on $\Pi^{\lambda,\alpha}$ such that its projection $\pi_{\mu}^{\lambda}(t)|_{\lambda/\mu}$ to $\lambda/\mu$ is still Markovian. The notable difference here is that for $\mu,\nu \subset \lambda^{\circ}$ it is possible for $(\pi^{\lambda}_{\mu}(t))_{t \geq 0}$ and $(\pi^{\lambda}_{\nu}(t))_{t \geq 0}$ to be different Markov processes. For instance, $\pi^{\lambda}_{\mu}(t)|_{\lambda/\nu}$ and $\pi^{\lambda}_{\nu}(t)|_{\lambda/\mu}$ might not be Markov processes even though we prove that $\pi^{\lambda}_{\mu}(t)|_{\lambda/\mu}$ and $\pi^{\lambda}_{\nu}(t)|_{\lambda/\nu}$ will always be Markov processes. (On the other hand, we note that $\pi^{\lambda}_{\delta_r}(t)$ and $\pi^{\lambda}_{\delta_s}(t)$ are the same process for all $\delta_r, \delta_s \subset \lambda^{\circ}$.) Of course, as $q \to 1$ the Markov processes $\pi^{\lambda}_{\mu}((1-q)^{-2}t)$ and $\pi^{\lambda}_{\nu}((1-q)^{-2}t)$ will both become the process $\pi^{\lambda}(t)$ from \cite{neil} for every $\mu,\nu \subset \lambda^{\circ}$. Our main result is the following,
\begin{theorem}\label{bigtheorem}
    For every skew shape $\lambda/\mu$ with $\mu \subset \lambda^{\circ}$ there exists a Markov process $\pi(t)$ on $\Pi^{\lambda,\alpha}$ such that if the following holds:
    \begin{itemize}
        \item $\pi$ evolves according to $G^{\lambda,\alpha}_{\mu}$,
        \item Initially, the conditional law of $\pi(t)$ is proportional to
    \[\begin{split} \widetilde{W}_{\lambda,\mu}(\pi,z;q) :=& q^{D_{\lambda,\mu}(\pi,z;q)} \prod_{(i,j) \in \mu} z_i^{\pi_{ij}}\binom{\pi_{ij}}{\pi_{i,j-1}}_q \frac{(q)_{\pi_{ij}}}{(q)_{\pi_{i-1,j}}(qz_{i-1,i+j-1})_{\pi_{ij}-\pi_{i-1,j}}} \\&\prod_{\substack{(i,j) \in \mu\\ (i,j+1) \in \lambda/\mu }} \binom{\pi_{i,j+1}}{\pi_{ij}}_q \prod_{\substack{(i,j) \in \mu \\ (i+1,j) \in \lambda/\mu}} \frac{(qz_{i,i+j})_{\pi_{i+1,j}}}{(q)_{\pi_{i,j}}(qz_{i,i+j})_{\pi_{i+1,j}-\pi_{ij}}}\end{split} \]
    with $D_{\lambda,\mu}=\displaystyle \sum_{(i,j) \in \mu} \pi_{ij}^2 - \sum_{\substack{(i,j) \in \mu \\ (i+1,j-1) \in \mu}} \pi_{ij}\pi_{i+1,j-1} - \sum_{\substack{(i,j) \in \lambda/\mu \\ (i+1,j-1) \in \mu}} \pi_{ij}\pi_{i+1,j-1}-\sum_{\substack{(i,j) \in \mu\\(i+1,j-1) \in \lambda/\mu}} \pi_{ij}\pi_{i+1,j-1}$.
    \end{itemize}
    Then the boundary values $\pi(t)|_{\lambda/\mu}$ will evolve as a Markov process on $\Pi^{\lambda/\mu, \alpha}$ with an explicit generator given in \eqref{anotherok}.
\end{theorem}
The case of Theorem \ref{bigtheorem} with $\lambda=\delta_{r+1}, \mu=\delta_r,$ and arbitrary parameters $\alpha_1,\alpha_2,\dots$ is done in Section $4$ using a generalization of our proof of Theorem \ref{smalltheorem}. In Subsection $5.1$ we prove Theorem \ref{bigtheorem} for general $\lambda \supset \delta_{r+2},\mu=\delta_{r+1}$ and trivial parameters $\alpha_1=\alpha_2=\dots=0$. Finally, in Subsection $5.2$ we will give the full proof of Theorem \ref{bigtheorem} with one of the key lemmas being proven in Appendix $A$. Prior to this, we will give some background in Section $2$ on the quantum difference Toda Hamiltonian and previous developments.\newline
\indent We end this section with an outline of the proof in Appendix A. There are five rate functions \[\widetilde{b}_{v}(\pi),\widetilde{b}_v(\sigma),b'_v(\pi),\hat{b}_v(\pi), \hat{b}'_v(\pi)\] which will be present in Appendix A. The function $\widetilde{b}_v(\pi)$ will be the rates initially described in Section 1.1 regarding the Markov process considered in Theorem 1.2. We note that these rates $\widetilde{b}_v(\pi)$ are dependent on the boundary of the skew shape $\lambda/\mu$. As usual, $\widetilde{b}_v(\sigma)$ will be the same rates except for the array $\sigma=(\sigma_{ij})$ given by $\sigma_{ij}=\pi_{ij}$ if $(i,j) \in \lambda/\mu$ and $\sigma_{ij}=0$ if $(i,j) \not\in \lambda/\mu$. The function $b'_{v}(\pi)$ is the quotient $R_{\pi_{v}}[\widetilde{b}_{v}\widetilde{W}_{\lambda,\mu}]/\widetilde{W}_{\lambda/\mu}$ where $\widetilde{W}_{\lambda,\mu}$ is the weight defined in the statement of Theorem 1.2. Finally, $\hat{b}_v,\hat{b}'_v$ are the rates described in Section 1.1 in regards to the Markov process considered in Theorem 1.2 except with $\mu=\emptyset$ trivial. The core of the proof comes down to considering the differences $\widetilde{b}_w(\pi)-\widetilde{b}_w(\sigma)$ and $b'_w(\pi)-\widetilde{b}_w(\pi)$ along certain hook shapes. Specifically, one can embed $\mu$ into a staircase shape $\hat{\mu}$---which one can assume is large enough to contain $\lambda$ as well---and the skew shape $\hat{\mu}/\mu$ can be thought of as several layers of hook shapes stacked on top of each other. There are several nonequivalent ways to realize $\hat{\mu}/\mu$ as several hook shapes glued together, but by fixing certain conventions (which is what we spend the beginning of Appendix A doing) one can make this construction unique. One can then consider all the hook shapes from this construction of $\hat{\mu}/\mu$ which are adjacent to $\mu$, these are referred to as the hook shape $\hat{F}_{ij}$ where $(i,j)$ is its origin in $\hat{\mu}/\mu$. In many of our calculations we will not need to sum over the entire hook shape $\hat{F}_{ij}$, and it will sometimes suffice to sum over a smaller hook shape $F_{ij}$. This smaller hook shape $F_{ij}$ will sit on top of a hook shape $G_{ij}$ which is contained entirely in $\mu$. Similarly, there exists a larger hook $\hat{G}_{ij}$ in $\hat{\mu}$ which contains $G_{ij}$ and is adjacent to $\hat{F}_{ij}$, here the situation is the same in the sense that we will often only need to sum over $G_{ij}$. To prove Theorem 1.1 the only novel computations are summation of $\widetilde{b}_w(\pi)-\widetilde{b}_w(\sigma)$ over $w \in F_{ij}$ and summation of $b'_w(\pi)-\widetilde{b}_w(\pi)$ over $w \in G_{ij}$, with various boundary terms being included. The rest of the proof amounts to comparing these quantities to summations over $\hat{b}'_v(\pi)-\hat{b}_v(\pi)$ for $v$ indexing points in a staircase shape, see \eqref{final5} for more details on this. As one can see in \eqref{big1}, summing over the vertical direction of a hook shape cancels out almost entirely with the exception of certain boundary terms. On the other hand, \eqref{big2} shows that summing over the horizontal direction of a hook shape results in boundary terms alongside extra contributions coming from points in the horizontal leg of the hook. These extra contributions are exactly what determine the form of $V_{\lambda,\mu}(\sigma)$ in \eqref{final6} and its dependence on the set $\mathrm{Hor}_{\mu}(\lambda)$. The remainder of the proof deals with how hook shapes glue together to form further cancellations, dealing with boundary terms, and tracking certain over-counting.
\subsection{Notation} The set $\mathbb{N}$ denotes the nonnegative integers. For $a \in \mathbb{C}\backslash \{1\}$, the $q$-Pochhammer symbol $(a;q)_n$ is defined via
\[(a;q)_n = \prod_{k=0}^{n-1} (1-aq^k).\]
When $a=q$ we write
\[(q)_n := (q;q)_n = \prod_{k=1}^n (1-q^k).\]
Moreover, $\binom{n}{k}_q = \frac{(q)_n}{(q)_k(q)_{n-k}}$ are the $q$-binomial coefficients. We note the standard scaling limits
\[\limmy{q}{1}{\frac{(q^a;q)_n}{(1-q)^n}}=(a)_n, \text{ }\text{ }\text{ }\text{ }\text{ }\text{ }\text{ } \limmy{q}{1}{\binom{n}{k}_q}=\binom{n}{k}\]
where $(a)_n := (a+1)(a+2)\dots(a+n-1)$ is the Pochhammer symbol and $\binom{n}{k}$ are the standard binomial coefficients. For integers $a<b$ the notation $\llbracket a, b\rrbracket$ will denote the set $\{a,a+1,\dots,b-1,b\}$.
\subsection{Acknowledgements}
The author is grateful to Neil O'Connell for many useful discussions, suggestions, and valuable comments. The author thanks Brian Rider for insightful conversations, their support, and their valuable criticisms of the first versions of this paper.
\section{Background}
Gerasimov, Kharchev, Lebedev, and Oblezin proved in \cite{GKLO} that certain eigenfunctions of
\begin{equation}\label{ok15}
    \mathfrak{h}^r=-\frac{1}{2}\sum_{i=1}^{r+1} \frac{\partial^2}{\partial x_i^2}+\sum_{i=1}^r e^{x_i-x_{i+1}}
\end{equation}
have an integral representation, and this integral relation can be understood as an intertwining relation between the Hamiltonians \eqref{ok15} for varying $r \geq 1$. These intertwinings were interpreted probabilistically in \cite{neil2} through exponential functionals of Brownian motion. The eigenfunctions that were understood in \cite{GKLO} and \cite{neil2} are considered "class one" Whittaker functions. There are another class of eigenfunctions of \eqref{ok15} which are known as "fundamental" Whittaker functions. Ishii and Stade found a recursive formula for the series coefficients of fundamental Whittaker functions, which was used in \cite{neil} to give a probabilistic interpretation to the this class of Whittaker functions in terms of a Markov chain on reverse plane partitions. For both classes of eigenfunctions of \eqref{ok15}, these results were proven by building a Markov process on triangular arrays and considering the projection to its outermost row---which again was a Markov process.\newline
\indent Define the $q$-shift operator $\mathfrak{D}_if(y_1,\dots,y_i,\dots,y_r)=f(y_1,\dots,qy_i,\dots,y_r)$. Our goal is to extend the above story to the quantum difference Toda hamiltonian \eqref{ok16} which has been studied in \cite{etingof}, \cite{feigin}, and \cite{givental}. This is the quantum difference Toda Hamiltonian for type $A$, whose eigenfunctions were constructed by \cite{etingof} and \cite{feigin} using Whittaker vectors of Verma modules coming from the representation theory of quantum groups of the Lie algebra $\mathfrak{sl}_{r+1}$. To see the connection between \eqref{ok15} and \eqref{ok16}, set $q=e^{\hbar}$ and $y_i=e^{x_i-x_{i+1}}$, then Remark $(2)$ in the introduction of \cite{givental} states
\begin{equation}\label{ok21}
    \mathfrak{H}_q^r = \hbar \sum_{i=1}^{r+1} \frac{\partial}{\partial x_i} + \bigg[\frac{\hbar^2}{2}\sum_{i=1}^{r+1} \frac{\partial^2}{\partial x_i^2} - \sum_{i=1}^r e^{x_i-x_{i+1}}\bigg] + O(\hbar^3).
\end{equation}
In the paper \cite{GLO} and its sequels, Gerasimov, Lebedev, and Oblezin found a set of eigenfunctions for \eqref{ok16} which are considered a $q$-deformation of class one Whittaker functions. These eigenfunctions correspond to Markov processes on triangular arrays, with formulas for these processes and their further probabilistic applications being detailed in \cite{macdonald}.\newline
\indent The set of eigenfunctions we choose to study are a $q$-deformation of the fundamental Whittaker functions, and we use the framework laid out in \cite{neil} to do this. To explain, let
\[
    \phi(y_1,\dots,y_r;q)=\sum_{n \in \mathbb{N}^r} a_r(n;q) y^n
\]
denote a series solution to the eigenvalue equation $\mathfrak{H}_q \phi=0$. As was pointed out in Appendix A of \cite{feigin}, one can verify that the coefficients $a_r(n;q)$ must satisfy a difference equation
\begin{equation}\label{diffeq}
    \sum_{i=0}^r (q^{n_{i+1}-n_i}-1)a_r(n;q)=\sum_{i=1}^r q^{n_{i+1}-n_i}a_r(n-e_i;q).
\end{equation}
A reverse plane partition filling $\delta_{r+1}$ is a $3$d diagram $(\pi_{ij})_{i+j-1\leq r}$ satisfying $\pi_{ij}\leq \mathrm{min}\{\pi_{i+1,j},\pi_{i,j+1}\}$ with $\pi_{ij} \geq 0$. Denote the set of reverse plane partitions filling $\delta_{r+1}$ by $\Pi^r$. In the paper \cite{toda}, the author determined the coefficients $a_r(n;q)$ as sums over triangular arrays $\underline{m}=(m_{k,i})_{1 \leq i \leq k \leq r}$ of nonnegative integers with $m_{r,i}=n_i$ and $m_{k,i} \leq m_{k+1,i}, m_{k+1,i+1}$, but one can readily convert these triangular arrays into reverse plane partitions. In terms of reverse plane partitions, Labelle's explicit combinatorial formulas from \cite{toda} for the coefficients $a_r(n;q)$ can be written as
\begin{equation}\label{labelle}
    a_r(n;q) = \frac{1}{(q)_n^2}\sum_{\pi \in \Pi^r_n} q^{\displaystyle \sum_{k=1}^{r-1}\sum_{i=1}^k \pi_{ij}(\pi_{ij}-\pi_{i+1,j-1})} \prod_{1 \leq i+j-1\leq r} \binom{\pi_{i+1,j}}{\pi_{ij}}_q\binom{\pi_{i,j+1}}{\pi_{ij}}_q
\end{equation}
where $\Pi^r_n$ is the set of reverse plane partitions $(\pi_{ij})_{2 \leq i + j \leq r+1}$ for which $(\pi_{i,r-i+1})_{i=1}^r=n$ is true. In order to parallel the developments in \cite{neil}, it is convenient to write down the following equivalent form
\begin{equation}\label{ok9}
    a_r(n;q)=\sum_{\pi \in \Pi^r_n} q^{\displaystyle \sum_{k=1}^{r-1} \sum_{i=1}^k \pi_{ij}(\pi_{ij}-\pi_{i+1,j-1})} \prod_{1 \leq i+j-1\leq r} \frac{1}{(q)_{\pi_{ij}-\pi_{i-1,j}}(q)_{\pi_{ij}-\pi_{i,j-1}}}.
\end{equation}
This form is obtained by reorganizing the products of $q$-binomial coefficients in \eqref{labelle}, followed by dividing the numerators out completely starting with the initial $1/(q)_n^2$ factor.\newline
\indent The coefficients $a_r(n;q)$ satisfy a natural recursion, for $n \in \mathbb{N}^r$ and $k\in \mathbb{N}^{r-1}$ define
\[
    q_r(n,k)=\prod_{i=1}^r \frac{1}{(q)_{n_i-k_i}(q)_{n_i-k_{i-1}}}
\]
so that
\begin{equation}\label{recursion}
    a_r(n;q)= \sum_{k \leq n} q^{\displaystyle \sum_{i=1}^{r-1} k_i(k_i-k_{i+1})} q_r(n,k) a_{r-1}(k;q).
\end{equation}
We will make essential use of \eqref{recursion}.
\section{Intertwinings and Markov-Doob Integrability}
Denote by $L_k$ and $R_k$ the shift operators for functions $f$ on $\mathbb{N}$ via
\[
    (L_kf)(k)=\begin{cases}
        f(k-1), &\text{ }k >0\\
        0, &\text{ otherwise.}
    \end{cases}
\]
and $(R_kf)(k)=f(k+1)$. The difference equation \eqref{diffeq} may be written as $h^ra_r=0$ where
\[
    h^r = \sum_{i=0}^r \bigg[q^{n_{i+1}-n_i}L_{n_i} + (1-q^{n_{i+1}-n_i})\bigg].
\]
Moreover, we define
\[
    (q_rf)(n)=\sum_{k \leq n} q^{\displaystyle \sum_{i=1}^{r-1} k_i(k_i-k_{i+1})}q_r(n,k)f(k)
\]
for functions $f$ on $\mathbb{N}^{r-1}$.\newline
\indent Everything in this paper works because of an intertwining relation between the operators $h^r$ and the operator $q_r$ determining the recursion of the $a_r$ coefficients. This is also where one of the first differences with \cite{neil} appears, where our $q_r$ operator includes not only the $q$-deformation of the $\frac{1}{(\pi_{ij}-\pi_{i-1,j})!(\pi_{ij}-\pi_{i,j-1})!}$ factors but also an extra $q^{\sum_{i=1}^{r-1} k_i(k_i-k_{i+1})}$ factor which should disappear in the $q\to 1$ limit. This deviation affects the proof of our first intertwining, because computing the adjoint $(h^{r}_n)^*$ requires an extra step and we must furthermore consider this formal adjoint with respect to the inner product
    \begin{equation}\label{innerproduct}
        \langle f,g \rangle_{q,r} = \sum_{n \in \mathbb{N}^r} q^{\displaystyle \sum_{i=1}^r n_i(n_i-n_{i+1})}f(n)g(n).
    \end{equation}
\begin{theorem}
    The following intertwining relation holds:
    \[
        h^r \circ q_r = q_r \circ h^{r-1}.
    \]
\end{theorem}
\begin{proof}
    First, note that \[R_{n_i} q^{\displaystyle \sum_{j=1}^r n_j(n_j-n_{j+1})}=q^{\displaystyle \sum_{j=1}^r n_j(n_j-n_{j+1})} q^{2n_i-n_{i+1}-n_{i-1}+1},\]
    and so:
    \[
        \begin{split}
            \langle q^{n_{i+1}-n_i}L_{n_i}f,g\rangle_{q,r} &= \sum_{n \in \mathbb{N}^r} q^{\displaystyle \sum_{j=1}^r n_j(n_j-n_{j+1})} q^{n_{i+1}-n_i}(L_{n_i}f)(n)g(n)\\ &= \sum_{n \in L_{n_i}\mathbb{N}^r} q^{\displaystyle \sum_{j=1}^r n_j(n_j-n_{j+1})} q^{2n_i-n_{i+1}-n_i+1} q^{n_{i+1}-(n_i+1)}f(n) (R_{n_i}g)(n) \\
            &= \sum_{n \in \mathbb{N}^r} q^{\displaystyle \sum_{j=1}^r n_j(n_j-n_{j+1})} q^{n_i-n_{i-1}} f(n) (R_{n_i}g)(n).
        \end{split}
    \]
    This gives us that
    \[
        (h^r_n)^* = \sum_{i=0}^r (1-q^{n_{i+1}-n_i})+\sum_{i=1}^rq^{n_i-n_{i-1}}R_{n_i},
    \]
    and one can readily confirm that
    \[
        h^r_n q_r(n,k)=-q_r(n,k)\sum_{i=0}^r \bigg[ (q^{n_{i+1}-n_i}-1)-q^{n_{i+1}-n_i}(1-q^{n_i-k_i})(1-q^{n_i-k_{i-1}}) \bigg],
    \]
    \[
        (h^{r-1}_k)^* q_r(n,k)=-q_r(n,k)\sum_{i=0}^{r-1} \bigg[ (q^{k_{i+1}-k_i}-1)-q^{k_i-k_{i-1}}(1-q^{n_i-k_i})(1-q^{n_{i+1}-k_i})\bigg].
    \]
    At this point, Lemma $5.1$ of \cite{toda} allows us to state the equality
    \[
        h^r_n q_r(n,k)=(h^{r-1}_k)^* q_r(n,k).
    \]
    To complete the proof observe that
    \[
        \begin{split}
            [(h^r \circ q_r)f](n) = \langle h_n^r q_r(n,\cdot),f\rangle_{q,r-1} &= \langle (h^{r-1})^* q_r(n,\cdot),f\rangle_{q,r-1} \\
            &= \langle q_r(n,\cdot),h^{r-1}f\rangle_{q,r-1}=[(q_r \circ h^{r-1})f](n).
        \end{split}
    \]
\end{proof}
Next introduce the Doob transform
\begin{equation}\label{doobtransform}
    L^r = a_r(n;q)^{-1} \circ h^r \circ a_r(n;q) = \sum_{i=1}^r q^{n_{i+1}-n_i} \frac{(L_{n_i}a_r)(n;q)}{a_r(n;q)} D_{n_i}f,
\end{equation}
where the rightmost equality follows from $h^r a_r=0$. Define
\begin{equation}\label{ok6}
    w_r(\pi;q)= \prod_{1 \leq i +j -1 \leq r}\frac{1}{(q)_{\pi_{ij}-\pi_{i-1,j}}(q)_{\pi_{ij}-\pi_{i,j-1}}}
\end{equation}
and observe that the recursion $w_r(\pi;q)=q_r(n,k)w_{r-1}(\pi|_{\delta_{r}};q)$ holds where $\delta_r=(r-1,r-2,\dots,2,1)$ is a staircase shape given by removing the outermost boxes in $\delta_{r+1}$. The measure on $\Pi^r_n$ given by
\[
    K^r_n(\pi;q)=q^{\displaystyle \sum_{i=1}^{r-1}\sum_{j=1}^i \pi_{ij}(\pi_{ij}-\pi_{i+1,j-1})}\frac{w_r(\pi;q)}{a_r(n)}.
\]
is a probability measure.\newline
\indent Recall the generator 
    \[
        G^r = \sum_{1 \leq i +j -1 \leq r} q^{\pi_{i+1,j-1}-\pi_{ij}} b_{ij}(\pi;q) D_{\pi_{ij}}
    \]
from \eqref{introthingy1} in the Introduction, where $b_{ij}(\pi;q)=(1-q^{\pi_{ij}-\pi_{i-1,j}})(1-q^{\pi_{ij}-\pi_{i,j-1}})$. Let $\pi(t), t\geq 0$ be the Markov process on $\Pi^r$ with generator $G^r$.\newline
\indent Let
\begin{equation}
    (\Lambda F)(n) = \sum_{\pi \in \Pi^r_n} K^r_n(\pi;q) F(\pi)
\end{equation}
be a Markov operator defined for $F: \Pi^r_n \to \mathbb{R}$, we want to prove the intertwining 
\[
    L^r \Lambda = \Lambda G^r.
\]
In the interest of proving such an intertwining, we define an operator $\tilde{q}_r$ on $\mathbb{N}^r \times \mathbb{N}^{r-1}$ via
\[
    (\tilde{q}_rf)(n)=\sum_{k \leq n} q^{\displaystyle \sum_{i=1}^{r-1} k_i(k_i-k_{i+1})}q_r(n,k) f(n,k)
\]
and leave it to the reader to confirm that for any $F :\Pi^r_n \to \mathbb{R}$ we can write \begin{equation} \sum_{\pi \in \Pi^r_n}q^{\displaystyle \sum_{(i,j) \in \delta_{r}} \pi_{ij}(\pi_{ij}-\pi_{i+1,j-1})}w_r(\pi;q)F(\pi)=\tilde{q}_r \tilde{q}_{r-1} \dots \tilde{q}_1 F_0\end{equation} for some $F_0 : \mathbb{N}^r \times \mathbb{N}^{r-1} \times \dots \times \mathbb{N} \to \mathbb{R}$. Moreover, if for $n^{(j)} \in \mathbb{N}^{r-j+1}, j=1,\dots,i$ satisfying $n^{(1)} \geq \dots \geq n^{(i)}$ we let $\Pi^r_{n^{(1)},\dots,n^{(i)}}$ denote the set of $\pi \in \Pi^r_{n^{(1)}}$ with $(\pi_{jk})_{j + k + i-2 = r} = n^{(i)}$, then it can also be verified that
\[
    (\tilde{q}_{r-i+1}\dots \tilde{q}_1 F_0)(n^{(1)},\dots,n^{(i)}) = \sum_{\pi \in \Pi^r_{n^{(1)},\dots,n^{(i)}}} q^{\displaystyle \sum_{(i,j) \in \delta_{r-i+1}} \pi_{ij}(\pi_{ij}-\pi_{i+1,j-1})}w_{r-i+1}(\pi|_{\delta_{r-i+2}};q)F(\pi).
\]
\begin{theorem}\label{markovfunction}
    The following intertwining relation holds:
    \begin{equation}\label{bigintertwining}
        L^r \Lambda = \Lambda G^r.
    \end{equation}
    As a result of this, if we let $\pi(t),t \geq 0$ be a Markov process on $\Pi^r$ with generator $G^r$ and initial law $K^r_n$ for some $n \in \mathbb{N}$. Then the projection to the outermost diagonal $N(t)=(\pi_{1,r}(t),\dots,\pi_{r,1}(t))$ is a Markov process on $\mathbb{N}^r$ with generator $L^r$ and, for all $t > 0$, the conditional law of $\pi(t)$ given $\{N(s), s \leq t\}$ is $K^r_{N(t)}$.
\end{theorem}
\begin{proof}
    If we set $F_{r-i}=\tilde{q}_{r-i}\dots \tilde{q}_1 F_0$ then it can be readily observed that $
        a_r(n;q)\Lambda G^r F = \sum_{i=1}^r \tilde{q}_r \dots \tilde{q}_{r-i+1} \tilde{g}^{r-i+1}F_{r-i}$ where $\tilde{g}^r = \sum_{i+j-1=r} q^{n_{i+1}-n_i}(1-q^{n_i-k_i})(1-q^{n_i-k_{i-1}})D_{n_i}$. Thus we follow \cite{neil} and define $g^r = h^{r-1}_k + \tilde{g}^r,$
    noting that
    \[
        \begin{split}
            [h^r (\tilde{q}_r f)] (n) &= \sum_{k \leq n} q^{\displaystyle \sum_{i=1}^{r-1} k_i(k_i-k_{i+1})}h^r_n [q_r(n,k)f(n,k)] \\
            &= \sum_{k \leq n} q^{\displaystyle \sum_{i=1}^{r-1} k_i(k_i-k_{i+1})} \bigg[ (h^r_nq_r(n,k))f(n,k)+\sum_{i=1}^r q^{n_{i+1}-n_i}L_{n_i}q_r(n,k) D_{n_i}f(n,k)\bigg] \\
            &= \sum_{k \leq n} q^{\displaystyle \sum_{i=1}^{r-1} k_i(k_i-k_{i+1})} \bigg[q_r(n,k)h^{r-1}_k f(n,k) + \sum_{i=1}^r q^{n_{i+1}-n_i}L_{n_i}q_r(n,k)D_{n_i}f(n,k)\bigg] \\
            &=\sum_{k \leq n} q^{\displaystyle \sum_{i=1}^{r-1} k_i(k_i-k_{i+1})}q_r(n,k)g^r f \\
            &= [\tilde{q}_r(g^r f)](n)
        \end{split}
    \]
    because $\frac{L_{n_i}q_{r}(n,k)}{q_r(n,k)}=(1-q^{n_i-k_i})(1-q^{n_i-k_{i-1}})$. One can see that the relation $L^r \Lambda = \Lambda G^r$ is equivalent to
    \[
        h^r \sum_{\pi \in \Pi^r_n} q^{\displaystyle \sum_{k=1}^{r-1}\sum_{i=1}^k \pi_{ij}(\pi_{ij}-\pi_{i+1,j-1})}w_r(\pi;q)F(\pi)=\sum_{\pi \in \Pi^r_n} q^{\displaystyle \sum_{k=1}^{r-1}\sum_{i=1}^k \pi_{ij}(\pi_{ij}-\pi_{i+1,j-1})} w_r(\pi;q)(G^rF)(\pi).
    \]
    Starting from the left-hand-side we have
    \[
        \begin{split}
            h^r \sum_{\pi \in \Pi^r_n} q^{\displaystyle \sum_{k=1}^{r-1}\sum_{i=1}^k \pi_{ij}(\pi_{ij}-\pi_{i+1,j-1})}w_r(\pi;q)F(\pi) &= h^r \tilde{q}_r \tilde{q}_{r-1} \dots \tilde{q}_1 F_0 \\
            &= \tilde{q}_r g^r \tilde{q}_{r-1} \dots \tilde{q}_1 F_0 \\
            &= \tilde{q}_r h^{r-1} \tilde{q}_{r-1} \dots \tilde{q}_1 F_0 + \tilde{q}_r \tilde{g}^r \tilde{q}_{r-1} \dots \tilde{q}_1 F_0, 
        \end{split}
    \]
    iterating again gives
    \begin{equation}\begin{split}
        \tilde{q}_r h^{r-1} \tilde{q}_{r-1} \dots \tilde{q}_1 F_0 + \tilde{q}_r \tilde{g}^r \tilde{q}_{r-1} \dots \tilde{q}_1 F_0 &= \tilde{q}_r \tilde{q}_{r-1} g^{r-1}\tilde{q}_{r-2} \dots \tilde{q}_1F_0 + \tilde{q}_r \tilde{g}^r \tilde{q}_{r-1} \dots \tilde{q}_1 F_0 \\
            &= \tilde{q}_r \tilde{q}_{r-1} (h^{r-2} + \tilde{g}^{r-1} )\tilde{q}_{r-2} \dots \tilde{q}_1F_0 + \tilde{q}_r \tilde{g}^r \tilde{q}_{r-1} \dots \tilde{q}_1 F_0 \\
            &=\tilde{q}_r\tilde{q}_{r-1}h^{r-2}\tilde{q}_{r-2} \dots \tilde{q}_1F_0+ \tilde{q}_r\tilde{q}_{r-1}\tilde{g}^{r-1}\tilde{q}_{r-2}\dots\tilde{q}_1F_0+\tilde{q}_r\tilde{g}^r\tilde{q}_{r-1}\dots\tilde{q}_1F_0.
    \end{split}\end{equation}
    Iterating further will give
    \begin{equation}\begin{split}
        h^r \sum_{\pi \in \Pi^r_n} q^{\displaystyle \sum_{k=1}^{r-1}\sum_{i=1}^k \pi_{ij}(\pi_{ij}-\pi_{i+1,j-1})}w_r(\pi;q)F(\pi) &=\tilde{q}_r\dots\tilde{q}_1 h^0 F_0 + \sum_{i=1}^r\tilde{q}_r\dots \tilde{q}_{r-i+1}\tilde{g}^{r-i+1}F_{r-i}\\
            &= \sum_{\pi \in \Pi^r_n} q^{\displaystyle \sum_{k=1}^{r-1}\sum_{i=1}^k \pi_{ij}(\pi_{ij}-\pi_{i+1,j-1})} w_r(\pi;q)(G^rF)(\pi)
    \end{split}\end{equation}
    and thus \eqref{bigintertwining} has been proven. The rest of the statement of Theorem $2.2$ holds by applying Equation $(1)$ in \cite{pitman}.
\end{proof}
\section{General Parameters}
We now consider a fixed sequence $\alpha_1,\alpha_2,\dots$ of nonnegative integers which are independent of $r$. Let $\alpha_{ij}=\alpha_i+\alpha_{i+1}+\dots+\alpha_j$ for $i \leq j$. Moreover, let $\beta_{ij}=\alpha_i+\dots+\alpha_{i+j-1}$ so that $z_{i-1,i+j-1}=q^{\beta_{ij}}$. We are now considering the operator
\[
    \mathfrak{H}_q^{\alpha} = \sum_{i=0}^r \bigg[\mathfrak{D}_{i+1}\mathfrak{D}_i^{-1}(z_{i,r}(1-y_i))-z_{i,r}\bigg]
\]
where $z_{i,j}=\prod_{k=i+1}^j z_k$ and $z_k=q^{\alpha_k}$ for $k=1,\dots,r$. In order for
\[
    \mathfrak{H}^{\alpha}_q \phi= 0, \text{ }\text{ }\text{ }\text{ }\text{ }\text{ }\text{ }\text{ }\text{ }\text{ } \phi(y,z;q)=\sum_{n \in \mathbb{N}^r} a_{r}(n,z;q) y^n
\]
to hold, the difference equation
\begin{equation}\label{ok5}
    \sum_{i=0}^r (q^{n_{i+1}-n_i}-1)z_{i,r}a_{r}(n,z;q) = \sum_{i=1}^r q^{n_{i+1}-n_i}z_{i,r}a_r(n-e_i,z;q) 
\end{equation}
must be satisfied. This motivates us to set
\[
    h^{r,\alpha}_n = \sum_{i=0}^r \bigg[ q^{n_{i+1}-n_i}z_{i,r}L_{n_i}+(1-q^{n_{i+1}-n_i})z_{i,r}\bigg].
\]
Before we record the form of $a_r(n,z)$, it is important to relax our notion of reverse plane partitions in order to take into account the drifts created by the $\alpha_1,\alpha_2,\dots$ parameters. Specifically, we now consider the set $\Pi^{r,\alpha}$ of non-negative integer arrays $(\pi_{ij})_{(i,j) \in \delta_{r+1}}$ with
\[
    \pi_{ij} \geq \mathrm{max}\{\pi_{i,j-1},\pi_{i-1,j} - \beta_{ij}\}.
\]
As usual, $\Pi^{r,\alpha}_n$ will denote elements $\pi \in \Pi^{r,\alpha}$ with $\pi_{i, r-i+1}=n_i$. \newline
\indent Now let
\begin{equation}\label{ok4}
    a_{r}(n,z;q) = \sum_{\pi \in \Pi^{r,\alpha}_n} q^{\displaystyle \sum_{(i,j) \in \delta_r} \pi_{ij}(\pi_{ij}-\pi_{i+1,j-1})} \prod_{(i,j) \in \delta_r} z_i^{\pi_{ij}}\prod_{(i,j) \in \delta_{r+1}} \frac{1}{(q)_{\pi_{ij}-\pi_{i,j-1}}(qz_{i-1,i+j-1})_{\pi_{ij}-\pi_{i-1,j}}}
\end{equation}
with
\[
    q_{r,\alpha}(n,k) = \prod_{i=1}^r \frac{1}{(q)_{n_i-k_i}(qz_{i-1,r})_{n_i-k_{i-1}}}.
\]
In order to confirm that \eqref{ok4} is a solution to \eqref{ok5}, we present a modified version of \cite{toda}'s Lemma 5.1:
\begin{lemma}
    Let $(a_i)_{i \in \mathbb{Z}}$, $(b_i)_{i \in \mathbb{Z}}$, $(c_i)_{i \in \mathbb{Z}}$ be sequences of integers, finitely many of which are nonzero. Then we have the identity
    \begin{equation}\begin{split}
        &\sum_{i} [(q^{a_{i+1}-a_i}-1)q^{c_i}-q^{a_{i+1}-a_i}q^{c_i}(1-q^{a_i-b_i})(1-q^{a_i-b_{i-1}}q^{c_{i-1}})]  \\
        &=\sum_{i} [(q^{b_{i+1}-b_i}-1)q^{c_i}-q^{b_i-b_{i-1}}q^{c_{i-1}}(1-q^{a_i-b_i})(1-q^{a_{i+1}-b_i}q^{c_i})]
    \end{split}
    \end{equation}
\end{lemma}
\begin{proof}
    Observe that, without the summations, the left-hand-side expands to
    \[-q^{a_{i+1}+a_i-b_i-b_{i-1}+c_i+c_{i-1}}+q^{a_{i+1}-b_{i-1}+c_i+c_{i-1}}+q^{a_{i+1}-b_i+c_i}-q^{c_i}\]
    and the right-hand-side expands to
    \[q^{b_{i+1}-b_i+c_i} - q^{c_i} - q^{b_i-b_{i-1}+c_{i-1}}+q^{a_i-b_{i-1}+c_{i-1}}+q^{a_{i+1}-b_{i-1}+c_i+c_{i-1}}-q^{a_i+a_{i+1}-b_i-b_{i-1}+c_i+c_{i-1}}.\]
    First of all, from the right-hand-side we note that
    \[
        \sum_{i=-M}^N (q^{b_{i+1}-b_i+c_i} - q^{b_i-b_{i-1}+c_{i-1}}) = q^{b_{N+1}-b_N+c_N}-q^{b_{-M}-b_{-M-1}+c_{-M-1}}=0
    \]
    where we assume that $\{-M,\dots,0,\dots,N\}$ contains the support of the sequences $(a_i)_{i \in \mathbb{Z}},(b_i)_{i \in \mathbb{Z}}, (c_i)_{i \in \mathbb{Z}}$. Moreover, both the left-hand-side and right-hand-side contain the same $-q^{a_{i+1}+a_i-b_i-b_{i-1}+c_i+c_{i-1}}$, $q^{a_{i+1}-b_{i-1}+c_i+c_{i-1}}$, and $-q^{c_i}$ terms, leaving us to confirm the identity
    \[
        \sum_{i=-M}^N q^{a_{i+1}-b_i+c_i} = \sum_{i=-M}^N q^{a_i-b_{i-1}+c_{i-1}},
    \]
    which is true because the right-hand-side is just the $i \mapsto i-1$ shift of the left-hand-side.
\end{proof}
The generalization of \eqref{innerproduct} to consider in the case with general parameters is
\[
    \langle f,g \rangle_{q,r,\alpha} = \sum_{n \in \mathbb{N}^r} q^{\sum_{i=1}^r n_i(n_i-n_{i+1})} \bigg(\prod_{i=1}^r z_i^{n_i}\bigg) f(n)g(n). 
\]
Under this inner product, we can see that
\[
    (h^{r,\alpha}_n)^* = \sum_{i=1}^r q^{n_i-n_{i-1}}z_{i-1,r}R_{n_i}+\sum_{i=0}^r(1-q^{n_{i+1}-n_i})z_{i,r}
\]
because $z_{i,r}R_{n_i} \prod_{i=1}^r z_i^{n_i}=z_{i,r}z_i \prod_{i=1}^r z_i^{n_i}=z_{i-1,r}\prod_{i=1}^r z_i^{n_i}$. Thus, using our Lemma $3.1$, we arrive at
\begin{equation}\label{ok19}
    h^{r,\alpha}_n q_r(n,k)=z_r(h^{r-1,\alpha}_k)^*q_r(n,k),
\end{equation}
which is sufficient to conclude the Toda recursion $h^{r,\alpha}a_r(n,z;q)=0$ via induction.\newline
\indent The Doob transform becomes
\[
    L^{r,\alpha} = a_r(n,z;q)^{-1} \circ h^{r,\alpha} \circ a_r(n,z;q) = \sum_{i=1}^r q^{n_{i+1}-n_i}z_{i,r} \frac{a_r(n-e_i,z;q)}{a_r(n,z;q)} D_{n_i}
\]
and the generator for the Markov process $\pi(t)$ on $\Pi^{r,\alpha}$ is
\[
    G^{r,\alpha} = \sum_{(i,j) \in \delta_{r+1}} q^{\pi_{i+1,j-1}-\pi_{ij}} z_{i,r}b_{ij}(\pi,z;q) D_{\pi_{ij}}
\]
where $b_{ij}(\pi,z;q)=(1-q^{\pi_{ij}-\pi_{i,j-1}})(1-q^{\pi_{ij}-\pi_{i-1,j}}z_{i-1,i+j-1})$. The corresponding probability measure on $\Pi^{r,\alpha}_n$ is now given by
\[
    K^r_n(\pi,z;q)=q^{\displaystyle \sum_{(i,j)\in \delta_r} \pi_{ij}(\pi_{ij}-\pi_{i+1,j-1})}\bigg(\prod_{(i,j)\in \delta_r} z_i^{\pi_{ij}}\bigg) \frac{w_r(\pi,z;q)}{a_r(n,z)}
\]
where
\[
    w_r(\pi,z;q)=\prod_{(i,j)\in\delta_{r+1}} \frac{1}{(q)_{\pi_{ij}-\pi_{i,j-1}}(qz_{i-1,i+j-1})_{\pi_{ij}-\pi_{i-1,j}}}
\]
The Markov function result Theorem \ref{markovfunction} extends over to the general parameter setting as well,
\begin{theorem}\label{thm42}
    If $\pi(t), t\geq 0$ is the Markov process on $\Pi^{r,\alpha}$ with generator $G^{r,\alpha}$ and initial law $K^r_n$ for some $n \in \mathbb{N}$. Then the projection to the outermost diagonal $N(t)=(\pi_{1,r}(t),\dots,\pi_{r,1}(t))$ is a Markov process on $\mathbb{N}^r$ with generator $L^{r,\alpha}$ and, for all $t > 0$, the conditional law of $\pi(t)$ given $\{N(s), s\leq t\}$ is $K^r_{N(t)}$.
\end{theorem}
\begin{proof}
The proof is almost the same as Theorem \ref{markovfunction}, except the presence of an extra $z_r$ in \eqref{ok19}  means that if we set $(\widetilde{q}_rf)(n)=\displaystyle \sum_{k \leq n} q^{\displaystyle \sum_{i=1}^{r-1}k_i(k_i-k_{i+1})} z^k q_r(n,k)f(n,k)$ then $h_r \circ \widetilde{q}_r =\widetilde{q}_r \circ g^r$ where
\[
    \widetilde{g}^r = \sum_{i +j -1 =r}q^{n_{i+1}-n_i} z_{i,r} (1-q^{n_i-k_i})(1-q^{n_i-k_{i-1}}z_{i-1,i+j-1})D_{n_i}
\]
and $g^r = z_rh^{r-1}_k+\widetilde{g}^r$. One can observe that \[\widetilde{q}_r \circ g^r \circ \widetilde{q}_{r-1}=\widetilde{q}_r\circ (z_rh^{r-1}+\widetilde{g}^r) \circ \widetilde{q}_{r-1} = z_r z_{r-1} \widetilde{q}_r \circ \widetilde{q}_{r-1} \circ h^{r-2}+z_r\widetilde{q}_r \circ \widetilde{q}_{r-1}\circ \widetilde{g}^{r-1}+\widetilde{q}_r \circ \widetilde{g}^r \circ \widetilde{q}_{r-1}  \]
where \[z_r\widetilde{q}_r \circ \widetilde{q}_{r-1}\circ \widetilde{g}^{r-1}=\widetilde{q}_r\circ \widetilde{q}_{r-1} \circ  \sum_{i+j-1=r-1} q^{k_{i+1}-k_i}z_{i,r} (1-q^{k_i-l_i})(1-q^{n_i-l_{i-1}}z_{i-1,i+j-1})D_{k_i}.\]
Thus, the iteration present in Theorem \ref{markovfunction} still works here, except we see $z_r \dots z_{r-i+1} \widetilde{g}^{r-i}$ at every stage instead of $\widetilde{g}^{r-i}$, hence what appears in $G^{r,\alpha}$ is $z_{i,r}$ rather than $z_{i,i+j-1}$. This is the only difference with Theorem \ref{markovfunction}, the rest of the proof proceeds in a straightforward manner.
\end{proof}
\section{Further Markovian Projections}
\indent Let $\lambda$ be a Young diagram which is not necessarily a staircase shape. Fix a Young diagram $\mu \subset \lambda^{\circ}$ and let $\tilde{\mu}$ be the extension of $\mu$ which contains $(i,j) \in \lambda$ for which $(i-1,j) \in \mu$ or $(i,j-1) \in \mu$. We consider the set $\Pi^{\lambda/\mu}$ of nonnegative integer arrays $(\sigma_{ij})_{(i,j) \in \lambda/\mu}$ for which $\sigma_{ij} \geq \mathrm{max}\{\sigma_{i,j-1},\sigma_{i-1,j}\}$ holds, with the convention that $\sigma_{ij}=0$ for $(i,j) \not\in \lambda/\mu$. We begin by proving Markov projections onto the boundary values at $\lambda/\mu$ for staircase shape $\mu=\delta_{r+1}$. Although this case provides essentialy no new content for us, we choose to use it as a warm-up for Subsection $5.2$'s proof concerning general skew shapes $\lambda/\mu$. The reader should note that the bulk of the details of the proof of Theorem \ref{bigtheorem}, also stated as Theorem \ref{bigresult} in Subsection $5.2$, are contained in Appendix $A$.
\subsection{Warm-Up}
For $\pi \in \Pi^{\lambda}=\Pi^{\lambda/\emptyset}$ and $\sigma \in \Pi^{\lambda/\mu}$ define
\[\widetilde{b}_{ij}(\pi;q)=q^{\pi_{i+1,j-1}-\pi_{ij}}(1-q^{\pi_{ij}-\pi_{i,j-1}})(1-q^{\pi_{ij}-\pi_{i-1,j}}),\text{ }\text{ }\text{ }\text{ }\text{ }\text{ }\text{ } \widetilde{b}_{ij}(\sigma;q)=q^{\sigma_{i+1,j-1}-\sigma_{ij}}(1-q^{\sigma_{ij}-\sigma_{i,j-1}})(1-q^{\sigma_{ij}-\sigma_{i-1,j}}).\]
Let $\mu=\delta_{r+1}$ be a staircase shape and assume that $\delta_{r+2} \subset \lambda$, we say $(u,w)\rightrightarrows (i,j)$ if $\{u,w\}=\{(i-1,j),(i,j-1)\}$.
\begin{definition}
    The generator for the Markov process on $\Pi^{\lambda/\mu}$ is
    \[
        G^{\lambda/\mu} = \sum_{(i,j) \in \lambda/\mu} \widetilde{b}_{ij}(\sigma;q) D_{\sigma_{ij}}.
    \]
\end{definition}
For $\sigma \in \Pi^{\lambda/\mu}$ we consider the set $\Pi^{\lambda}_{\sigma}$ consisting of nonnegative integer arrays $\pi \in \Pi^{\lambda}$ for which $\pi|_{\lambda/\mu}=\sigma$. For each $\pi \in \Pi^{\lambda}_{\sigma}$ we may define the weight
\[
    W_{\lambda,\mu}(\pi;q) = \prod_{(i,j) \in \delta_{r+2}} \binom{\pi_{ij}}{\pi_{i,j-1}}_q \binom{\pi_{ij}}{\pi_{i-1,j}}_q
\]
and the corresponding probability measure on $\Pi^{\lambda/\mu}_{\sigma}$ is given by
\[
    K^{\lambda,\mu}_{\sigma}(\pi;q)=q^{\displaystyle\sum_{(i,j) \in \delta_{r+1}} \pi_{ij}(\pi_{ij}-\pi_{i+1,j-1})} \frac{W_{\lambda,\mu}(\pi;q)}{A_{\lambda,\mu}(\sigma;q)}
\]
where
\begin{equation}
    A_{\lambda,\mu}(\sigma;q) = \sum_{\pi\in\Pi^{\lambda}_{\sigma}}q^{\displaystyle\sum_{(i,j) \in \delta_{r+1}} \pi_{ij}(\pi_{ij}-\pi_{i+1,j-1})} W_{\lambda,\mu}(\pi;q).
\end{equation}
It will be convenient for us to use the notation
\[
    \widetilde{W}_{\lambda,\mu}(\pi;q) = q^{\displaystyle\sum_{(i,j) \in \delta_{r+1}} \pi_{ij}(\pi_{ij}-\pi_{i+1,j-1})} W_{\lambda,\mu}(\pi;q).
\]
This subsection's main result, which will prepare us for the proof in Subsection $5.2$, is
\begin{theorem}\label{bigdriftlessresult}
    Fix $\mu=\delta_{r+1}$ and a Young diagram $\lambda$ with $\mu \subset \lambda^{\circ}$. If $\pi(t)$ is the Markov process on $\Pi^{\lambda}$ with generator $G^{\lambda}$ and initial law, conditional on the values on $\lambda/\mu$, given by $K^{\lambda,\mu}_{\sigma}$ for some $\sigma \in \Pi^{\lambda/\mu}$, then $\sigma(t)=(\pi(t))_{\lambda/\mu}$ is a Markov process and, for all $t > 0$, the conditional law of $\pi(t)$ given $\{\sigma(s), s \leq t\}$ is $K^{\lambda,\mu}_{\sigma(t)}$.
\end{theorem}
Before we begin the proof of Theorem \ref{bigdriftlessresult}, we record two important lemmas.
\begin{lemma}\label{somelemma1} For $\mu=\delta_{r+1}$,
    \begin{equation}\label{ok27}\begin{split}
        \sum_{(i,j) \in \lambda/\mu} [\widetilde{b}_{ij}(\pi;q)-\widetilde{b}_{ij}(\sigma;q)]&=-\sum_{\substack{ v=(i,j)\in \lambda/\mu \\ (u,w)\rightrightarrows v \\ u \in \mu , x=(i+1,j-1) }} q^{\pi_x-\pi_u}(1-q^{\pi_u})(1-q^{\pi_v-\pi_w})\\&-\sum_{\substack{(i,j) \in \lambda/\mu \\ (i,j-1) \in \mu \text{ and } (i-1,j) \in \mu }} q^{\pi_{i+1,j-1}+\pi_{ij}-\pi_{i-1,j}-\pi_{i,j-1}} (1-q^{\pi_{i-1,j}})(1-q^{\pi_{i,j-1}})
    \end{split}
    \end{equation}
    where $(u,w) \rightrightarrows v$ and $(w,u) \rightrightarrows v$ are equivalent statements indexing different elements. When expanded,
    \begin{equation}\label{ok22}\begin{split}
        \sum_{(i,j) \in \lambda/\mu} [\widetilde{b}_{ij}(\pi;q)-\widetilde{b}_{ij}(\sigma;q)] &=(r+2)-q^{\pi_{1,r+1}}-\sum_{\substack{(i,j) \in \lambda/\mu \\ (i-1,j) \in \mu \text{ or } (i,j-1) \in \mu}} (1-q^{\pi_{i+1,j-1}})(1-q^{\pi_{ij}})\\&-\sum_{\substack{(i,j) \in \delta_{r+2}/\delta_{r+1} }} (q^{\pi_{i+1,j-1}-\pi_{i,j-1}}+q^{\pi_{i+1,j-1}-\pi_{i-1,j}}-q^{\pi_{i+1,j-1}+\pi_{ij}-\pi_{i,j-1}-\pi_{i-1,j}})
    \end{split}
    \end{equation}
\end{lemma}
\begin{proof}
    Observe that
    \[\begin{split}
        \sum_{\substack{(i,j) \in \lambda/\mu\\ (i+1,j-1) \not\in \mu }} [\widetilde{b}_{ij}&(\pi;q)-\widetilde{b}_{ij}(\sigma;q)]=\\&\sum_{\substack{(i,j) \in \lambda/\mu\\ (i+1,j-1) \not\in \mu \\ (i,j-1) \not\in \lambda/\mu \text{ or }(i-1,j) \not\in \lambda/\mu}} q^{\pi_{i+1,j-1}-\pi_{ij}}[(1-q^{\pi_{ij}-\pi_{i,j-1}})(1-q^{\pi_{ij}-\pi_{i-1,j}})-(1-q^{\pi_{ij}-\sigma_{i,j-1}})(1-q^{\pi_{ij}-\sigma_{i-1,j}})].
    \end{split}\]
    For $(i,j)\in \lambda/\mu$ define $c_{ij}(\pi;q)$ to be $-(1-q^{\pi_{i,j-1}})(1-q^{\pi_{ij}-\pi_{i-1,j}})$ if $(i,j-1) \not\in \lambda/\mu, (i-1,j) \in \lambda/\mu$, with $c_{ij}(\pi;q)=-(1-q^{\pi_{ij}-\pi_{i,j-1}})(1-q^{\pi_{i-1,j}})$ if $(i,j-1) \in \lambda/\mu, (i-1,j) \not\in \lambda/\mu$, and finally
    \[
        c_{ij}(\pi;q)=-q^{\pi_{i-1,j}}(1-q^{\pi_{i,j-1}})(1-q^{\pi_{ij}-\pi_{i-1,j}})-q^{\pi_{i,j-1}}(1-q^{\pi_{ij}-\pi_{i,j-1}})(1-q^{\pi_{i-1,j}})-q^{\pi_{ij}}(1-q^{\pi_{i,j-1}})(1-q^{\pi_{i-1,j}})
    \]
    if $(i,j-1) \not\in \lambda/\mu, (i-1,j) \not\in \lambda/\mu$. We have
    \[\begin{split}
        (1-q^{\pi_{ij}-\pi_{i,j-1}})(1-q^{\pi_{ij}-\pi_{i-1,j}})-(1-&q^{\pi_{ij}-\sigma_{i,j-1}})(1-q^{\pi_{ij}-\sigma_{i-1,j}})\\&=\begin{cases}
           q^{\pi_{ij}-\pi_{i,j-1}}c_{ij}(\pi;q) ,&\text{ if } (i,j-1) \not\in \lambda/\mu, (i-1,j) \in \lambda/\mu \\
           q^{\pi_{ij}-\pi_{i-1,j}}c_{ij}(\pi;q), &\text{ if } (i,j-1) \in \lambda/\mu, (i-1,j) \not\in \lambda/\mu \\
           q^{\pi_{ij}-\pi_{i-1,j}-\pi_{i,j-1}}c_{ij}(\pi;q), &\text{ if } (i,j-1) \not\in \lambda/\mu, (i-1,j) \not\in \lambda/\mu
        \end{cases}
    \end{split}
    \]
    One should recall that we are considering staircase shapes for $\mu$, so $(i,j-1) \not\in \lambda/\mu$ and $(i-1,j) \not\in \lambda/\mu$ is always the case we consider. For \eqref{ok22} we expand \eqref{ok27} and use
    \[
        2q^{\pi_x}-q^{\pi_x+\pi_v}=-(1-q^{\pi_x})(1-q^{\pi_v})+1+(q^{\pi_x}-q^{\pi_v}).
    \]
\end{proof}
Another lemma that will be key for us is
\begin{lemma}\label{ok20}
    For $\pi \in \Pi^{\lambda}_{\sigma}$ and $(i,j) \in \lambda/\mu$,
    \begin{equation}
        \widetilde{b}_{ij}(\sigma;q) L_{\sigma_{ij}}\widetilde{W}_{\lambda,\mu}(\pi;q)=\widetilde{b}_{ij}(\pi;q)\widetilde{W}_{\lambda,\mu}(\pi;q).
    \end{equation}
\end{lemma}
We are now ready to prove this subsection's result.
\begin{proof}[Proof of Theorem \ref{bigdriftlessresult}]
    We first define $\Lambda$ to operate on functions $F : \Pi^{\lambda}_{\sigma} \to \mathbb{R}$ via
    \[
        (\Lambda F)(\sigma) = \sum_{\pi \in \Pi^{\lambda}_{\sigma}} \widetilde{W}_{\lambda,\mu}(\pi;q) F(\pi).
    \]
    Let $C(\mu)$ denote the set of points $(i,j) \in \mu$ for which $(i+1,j) \not\in \mu$ and $(i,j+1) \not\in \mu$. Define $H^{\lambda/\mu} = G^{\lambda/\mu} + V_{\lambda,\mu}$ where
    \[
        V_{\lambda,\mu}(\sigma;q)=\sum_{(i,j) \in C(\mu)} (1-q^{\sigma_{i+1,j}})(1-q^{\sigma_{i,j+1}})
    \]
    Our first goal is to prove that $H^{\lambda/\mu}\circ \Lambda = \Lambda \circ G^{\lambda}$. Observe that
    \[
        \begin{split}
            (H^{\lambda/\mu} \Lambda F)(\sigma) &= \sum_{\pi \in \Pi^{\lambda}_{\sigma}} \bigg([H^{\lambda/\mu}\widetilde{W}_{\lambda,\mu}(\pi;q)]F(\pi)+\sum_{(i,j) \in \lambda/\mu} [L_{\sigma_{ij}}\widetilde{W}_{\lambda,\mu}(\pi;q)]\widetilde{b}_{ij}(\sigma;q) D_{\sigma_{ij}}F(\pi) \bigg) \\
            &= \sum_{\pi \in \Pi^{\lambda}_{\sigma}} \bigg([H^{\lambda/\mu}\widetilde{W}_{\lambda,\mu}(\pi;q)]F(\pi)+\widetilde{W}_{\lambda,\mu}(\pi;q)\sum_{(i,j) \in \lambda/\mu} \widetilde{b}_{ij}(\pi;q) D_{\sigma_{ij}}F(\pi) \bigg)
        \end{split}
    \]
    where the last line follows from Lemma \ref{ok20}. It remains to show that $H^{\lambda/\mu}\widetilde{W}_{\lambda,\mu}=(G^{\mu})' \widetilde{W}_{\lambda,\mu}$ where $(-)'$ denotes the formal adjoint with respect to the inner product 
    \[\langle f,g \rangle_{\mu} = \sum_{\pi \in \Pi^{\mu}} f(\pi)g(\pi).\]
    Note that $(-)'$ is different from the formal adjoints $(-)^*$ we took with respect to a $q$-inner product in Sections $3$ and $4$, because in this proof we have instead opted to absorb the extra $q$-factors into the weights $\widetilde{W}_{\lambda,\mu}(\pi;q)$ rather than the underlying inner product. Since $(G^{\mu})'f = \displaystyle \sum_{(i,j) \in \mu} B_{\pi_{ij}}[\widetilde{b}_{ij}(\pi;q)f]$ where $B_{k}=R_k-I$ is the forward difference operator, we must compute $B_{\pi_{ij}}[\widetilde{b}_{ij}(\pi;q)\widetilde{W}_{\lambda,\mu}(\pi;q)]$: For $(i,j) \in \mu$,
    \begin{equation}\label{ok30}
        R_{\pi_{ij}}[\widetilde{b}_{ij}(\pi;q) \widetilde{W}_{\lambda,\mu}(\pi;q)] = q^{\pi_{ij}-\pi_{i-1,j+1}}(1-q^{\pi_{i,j+1}-\pi_{ij}})(1-q^{\pi_{i+1,j}-\pi_{ij}})\widetilde{W}_{\lambda,\mu}(\pi;q)
    \end{equation}
    and so we set $b_{ij}'(\pi;q) := q^{\pi_{ij}-\pi_{i-1,j+1}}(1-q^{\pi_{i,j+1}-\pi_{ij}})(1-q^{\pi_{i+1,j}-\pi_{ij}})$. Thus,
    \[
        B_{\pi_{ij}}[\widetilde{b}_{ij}(\pi;q)\widetilde{W}_{\lambda,\mu}(\pi;q)]=[b_{ij}'(\pi;q)-\widetilde{b}_{ij}(\pi;q)]\widetilde{W}_{\lambda,\mu}(\pi;q).
    \]
    Moreover, one use Lemma \ref{ok20} to state
    \[
        \widetilde{b}_{ij}(\sigma;q) D_{\sigma_{ij}} \widetilde{W}_{\lambda,\mu}(\pi;q)=[\widetilde{b}_{ij}(\pi;q)-\widetilde{b}_{ij}(\sigma;q)]\widetilde{W}_{\lambda,\mu}(\pi;q).
    \]
    So the statement $H^{\lambda/\mu}\widetilde{W}_{\lambda,\mu}=(G^{\mu})' \widetilde{W}_{\lambda,\mu}$ is equivalent to
    \[
        \sum_{(i,j) \in \lambda/\mu} [\widetilde{b}_{ij}(\pi;q)-\widetilde{b}_{ij}(\sigma;q)] + V_{\lambda,\mu}(\sigma;q) = \sum_{(i,j) \in \mu} [b_{ij}'(\pi;q) - \widetilde{b}_{ij}(\pi;q)].
    \]
    For $\mu=\delta_{r+1}$ we define the set $\hat{\mu}=\mu \cup \{(i,0) \text{ }:\text{ }1 \leq i \leq r+1\} \cup \{(0,j) : 1 \leq j \leq r+1\} \cup \{(0,0)\}$, observing that $b_{ij}'(\pi;q) - \widetilde{b}_{ij}(\pi;q)=0$ for $(i,j) \in \hat{\mu}/ \mu$ and so
    \[\sum_{(i,j) \in \mu} [b_{ij}'(\pi;q)-\widetilde{b}_{ij}(\pi;q)]=\sum_{(i,j) \in \hat{\mu}}[b_{ij}'(\pi;q)-\widetilde{b}_{ij}(\pi;q)].\]
    The set $\hat{\mu}$ forms a Young diagram, and is not to be confused with the Young diagram $\widetilde{\mu}$. One can expand $b_{ij}'(\pi;q)-\widetilde{b}_{ij}(\pi;q)$ in order to write
    \[
        \sum_{(i,j) \in \hat{\mu}} [b_{ij}'(\pi;q) - \widetilde{b}_{ij}(\pi;q)]=(A)-(B)-(C)+(D)
    \]
    where
    \[(A)= \sum_{\substack{u=(i,j)\in \hat{\mu} \\ x=(i-1,j+1),y=(i+1,j-1)}} (q^{\pi_u-\pi_x}-q^{\pi_y-\pi_u})=0,\]
    \[\begin{split} (B)=\sum_{\substack{u = (i,j) \in \hat{\mu} \\ x=(i-1,j+1),y=(i+1,j-1) \\ v=(i,j+1),w=(i,j-1)}} (q^{\pi_v-\pi_x}-q^{\pi_y-\pi_w})&=\sum_{\substack{(i,j) \in \hat{\mu}\\ (i,j+1) \in \lambda/\mu \text{ or } (i+1,j) \in \lambda/\mu }} q^{\pi_{i,j+1}-\pi_{i-1,j+1}}-(r+2)\\
    &=1+q^{\pi_{1,r+1}}+\sum_{(i,j) \in \delta_{r+2}/\delta_{r+1}} q^{\pi_{i+1,j-1}-\pi_{i,j-1}}-1-(r+2)\end{split}\]
    \[\begin{split} (C)=\sum_{\substack{u=(i,j) \in \hat{\mu}, v=(i+1,j) \\ x=(i-1,j+1),y=(i+1,j-1)\\ w=(i-1,j)}} (q^{\pi_v-\pi_x}-q^{\pi_y-\pi_w})&=\sum_{\substack{(i,j) \in \hat{\mu} \\ (i,j+1) \in \lambda/\mu \text{ or } (i+1,j) \in \lambda/\mu}} q^{\pi_{i+1,j}-\pi_{i-1,j+1}}-(r+2)\\
    &=q^{\pi_{1,r+1}}+\sum_{(i,j) \in \delta_{r+2}/\delta_{r+1}} q^{\pi_{i+1,j-1}-\pi_{i-1,j}}-(r+2)\end{split},\]
    and
    \[\begin{split} (D)&=\sum_{\substack{(i,j) \in \hat{\mu}}} (q^{\pi_{i+1,j}+\pi_{i,j+1}-\pi_{ij}-\pi_{i-1,j+1}}-q^{\pi_{i+1,j-1}+\pi_{ij}-\pi_{i,j-1}-\pi_{i-1,j}})\\&=\sum_{\substack{(i,j) \in \hat{\mu} \\ (i,j+1) \in \lambda/\mu \text{ or } (i+1,j) \in \lambda/\mu}} q^{\pi_{i+1,j}+\pi_{i,j+1}-\pi_{ij}-\pi_{i-1,j+1}}-(r+2)\\
    &=q^{\pi_{1,r+1}}+\sum_{(i,j) \in \delta_{r+2}/\delta_{r+1}} q^{\pi_{i+1,j-1}+\pi_{ij}-\pi_{i,j-1}-\pi_{i-1,j}}\end{split}\]
    Therefore,
    \begin{equation}\label{ok23}\begin{split}
        \sum_{(i,j) \in \mu} [b_{ij}'(\pi;q) - \widetilde{b}_{ij}(\pi;q)] =& (r+2)-q^{\pi_{1,r+1}}\\&+\sum_{\substack{(i,j) \in \delta_{r+2}/\delta_{r+1} } }(q^{\pi_{i+1,j-1}+\pi_{ij}-\pi_{i,j-1}-\pi_{i-1,j}}- q^{\pi_{i+1,j-1}-\pi_{i,j-1}} -  q^{\pi_{i+1,j-1}-\pi_{i-1,j}})
    \end{split}
    \end{equation}
    Comparing \eqref{ok23} with \eqref{ok22} tells us that $\displaystyle \sum_{(i,j) \in \mu} [b_{ij}'(\pi;q) - \widetilde{b}_{ij}(\pi;q)]-\sum_{(i,j) \in \lambda/\mu} [\widetilde{b}_{ij}(\pi;q)-\widetilde{b}_{ij}(\sigma;q)]$ is equal to
    \[\begin{split} \sum_{\substack{(i,j) \in \lambda/\mu \\ (i-1,j) \not\in \lambda/\mu \text{ and } (i,j-1) \not\in \lambda/\mu}} (1-q^{\pi_{i+1,j-1}})(1-q^{\pi_{ij}})=V_{\lambda,\mu}(\sigma;q).
    \end{split}\]
    \indent This concludes our proof that $H^{\lambda/\mu} \circ \Lambda = \Lambda \circ G^{\lambda/\mu}$, and observe that in the process of proving this statement we have also shown that
    \[
        H^{\lambda/\mu}A_{\lambda,\mu}=0.
    \]
    Thus, 
    \[
        L^{\lambda/\mu}=A_{\lambda,\mu}^{-1} \circ H^{\lambda/\mu} \circ A_{\lambda,\mu}
    \]
    generates a Markov process on $\Pi^{\lambda/\mu}$. In fact, if we define
    \[
        (\widetilde{\Lambda}F)(\sigma) = \sum_{\pi \in \Pi^{\lambda}_{\sigma}} K^{\lambda,\mu}_{\sigma}(\pi;q) F(\pi)
    \]
    then
    \[
        L^{\lambda/\mu} \circ \widetilde{\Lambda} = \widetilde{\Lambda} \circ G^{\lambda}.
    \]
    Finally, this gives us the statement of the theorem, with $L^{\lambda/\mu}$ being the generator of $\sigma(t)$.
\end{proof}
\subsection{General Result}
\indent We now consider the case where $\mu$ is a fixed Young diagram, not necessarily of staircase shape, satisfying $\mu \subset \lambda^{\circ}$. We are again allowing the parameters $\alpha_1,\alpha_2,\dots$ to be arbitrary nonnegative integers. For the purposes of this subsection we will let $f_{ij}^+(\pi),f^-_{ij}(\pi), g_{ij}^+(\pi), g_{ij}^-(\pi), h_{ij}^+(\pi), h_{ij}^-(\pi),$ and $m_{ij}(\pi)$ be a collection of seven functions which will determine the rates of particles in various regions of $\lambda$, and by determining these functions we will be able to construct the stochastic dynamics on $\Pi^{\lambda,\alpha}$ which correspond to generator $G_{\mu}^{\lambda,\alpha}$ and weights $\widetilde{W}_{\lambda,\mu}(\pi,z;q)$ from Theorem \ref{bigtheorem}. \newline
\indent As usual, let $\Pi^{\lambda/\mu,\alpha}$ by the set of nonnegative integer arrays $(\sigma_{ij})_{(i,j) \in \lambda/\mu}$ for which $\sigma_{ij} \geq \mathrm{max}\{\sigma_{i,j-1},\sigma_{i-1,j}-\beta_{ij}\}$ holds. The set $\Pi^{\lambda,\alpha}$ will consist of non-negative integer arrays $(\pi_{ij})_{(i,j) \in \lambda}$ for which $\pi_{ij} \geq \mathrm{max}\{\pi_{i,j-1},\pi_{i-1,j}-\beta_{ij}\}$ holds. Moreover, let $\Pi^{\lambda,\alpha}_{\sigma}$ be the set consisting of $\pi \in \Pi^{\lambda,\alpha}$ for which $\pi|_{\lambda/\mu}=\sigma$. We first recall from Subsection $1.2$ the weights of arrays $\pi \in \Pi^{\lambda,\alpha}_{\sigma}$ for $\sigma \in \Pi^{\lambda/\mu,\alpha}$:
\begin{equation}\label{ok28}\begin{split}
    &\widetilde{W}_{\lambda,\mu}(\pi,z;q)=\\&q^{D_{\lambda,\mu}(\pi,z;q)} \prod_{(i,j) \in \mu} z_i^{\pi_{ij}}\binom{\pi_{ij}}{\pi_{i,j-1}}_q \frac{(q)_{\pi_{ij}}}{(q)_{\pi_{i-1,j}}(qz_{i-1,i+j-1})_{\pi_{ij}-\pi_{i-1,j}}} \prod_{\substack{(i,j) \in \mu\\ (i,j+1) \in \lambda/\mu }} \binom{\pi_{i,j+1}}{\pi_{ij}}_q \prod_{\substack{(i,j) \in \mu \\ (i+1,j) \in \lambda/\mu}} \frac{(qz_{i,i+j})_{\pi_{i+1,j}}}{(q)_{\pi_{i,j}}(qz_{i,i+j})_{\pi_{i+1,j}-\pi_{ij}}}
    \end{split}
\end{equation}
where 
\begin{equation}\begin{split}\label{ok29} 
D_{\lambda,\mu}(\pi,z;q)= \sum_{(i,j) \in \mu} \pi_{ij}^2 - \sum_{\substack{(i,j) \in \mu \\ (i+1,j-1) \in \mu}} \pi_{ij}\pi_{i+1,j-1} - \sum_{\substack{(i,j) \in \lambda/\mu \\ (i+1,j-1) \in \mu}} \pi_{ij}\pi_{i+1,j-1}-\sum_{\substack{(i,j) \in \mu\\(i+1,j-1) \in \lambda/\mu}} \pi_{ij}\pi_{i+1,j-1}\end{split}\end{equation}
For $u,v \in \lambda$ we say $u \sim_{\mu} v$ if $u,v \not\in \lambda/\mu$ or $u,v \in \lambda/\mu$. We use the convention that $u \sim_{\mu} v$ is always true if $u \in (\mathbb{N}\times \{0\}) \cup (\{0\} \times \mathbb{N})$ or $v\in (\mathbb{N}\times \{0\}) \cup (\{0\} \times \mathbb{N})$.\newline
\indent For $(i,j) \in \lambda$ write \[\begin{split} \widetilde{b}_{ij}(\pi,z;q) =&z_{i,\ell(\lambda)}(1-q^{\pi_{ij}-\pi_{i,j-1}})(1-z_{i-1,i+j-1}q^{\pi_{ij}-\pi_{i-1,j}})\times\\&\begin{cases} q^{\pi_{i+1,j-1}-\pi_{ij}} &\text{ if } (i-1,j+1) \sim_{\mu} (i,j) \sim_{\mu} (i+1,j-1), (i,j) \not\in \mathrm{Vert}_{\mu}(\lambda)\\
q^{\pi_{i+1,j-1}-\pi_{ij}}q^{m_{ij}(\pi)},&\text{ if }(i-1,j+1) \sim_{\mu} (i,j) \sim_{\mu} (i+1,j-1), (i,j) \in \mathrm{Vert}_{\mu}(\lambda)\\ q^{\pi_{i+1,j-1}-\pi_{ij}}q^{f_{ij}^+(\pi)},&\text{ if } (i-1,j+1) \sim_{\mu} (i,j) \not\sim_{\mu} (i+1,j-1),(i,j) \in \lambda/\mu\\
q^{\pi_{i+1,j-1}-\pi_{ij}}q^{f_{ij}^-(\pi)}, &\text{ if } (i-1,j+1) \sim_{\mu} (i,j) \not\sim_{\mu} (i+1,j-1),(i,j) \not\in \lambda/\mu\\ 
q^{\pi_{i-1,j+1}-\pi_{ij}}q^{g_{ij}^+(\pi)}, &\text{ if } (i-1,j+1) \not\sim_{\mu} (i,j) \sim_{\mu} (i+1,j-1), (i,j) \in \lambda/\mu\\ q^{\pi_{i+1,j-1}-\pi_{ij}}q^{g_{ij}^-(\pi)}, &\text{ if } (i-1,j+1) \not\sim_{\mu} (i,j) \sim_{\mu} (i+1,j-1), (i,j) \not\in \lambda/\mu \\ q^{\pi_{i+1,j-1}+\pi_{i-1,j+1}-\pi_{ij}}q^{h_{ij}^+(\pi)}, &\text{ if } (i-1,j+1) \not\sim_{\mu} (i,j) \not\sim_{\mu} (i+1,j-1), (i,j) \in\lambda/\mu\\
q^{\pi_{i+1,j-1}-\pi_{ij}}q^{h_{ij}^-(\pi)}, &\text{ if } (i-1,j+1) \not\sim_{\mu} (i,j) \not\sim_{\mu} (i+1,j-1), (i,j) \not\in \lambda/\mu\end{cases}\end{split}\]
and for $(i,j) \in \lambda/\mu$ write \[\begin{split}\widetilde{b}_{ij}(\sigma,z;q)=&z_{i,\ell(\lambda)}(1-q^{\sigma_{ij}-\sigma_{i,j-1}})(1-z_{i-1,i+j-1}q^{\sigma_{ij}-\sigma_{i-1,j}})\times \\&\begin{cases} q^{\sigma_{i+1,j-1}-\sigma_{ij}}, &\text{ if } (i-1,j+1) \sim_{\mu} (i,j) \sim_{\mu} (i+1,j-1), (i,j) \not\in \mathrm{Vert}_{\mu}(\lambda) \\q^{\sigma_{i+1,j-1}-\sigma_{ij}}q^{m_{ij}(\sigma)}, &\text{ if }(i-1,j+1) \sim_{\mu} (i,j) \sim_{\mu} (i+1,j-1), (i,j) \in \mathrm{Vert}_{\mu}(\lambda)\\ q^{\sigma_{i+1,j-1}-\sigma_{ij}}q^{f_{ij}^+(\sigma)}, &\text{ if } (i-1,j+1) \sim_{\mu} (i,j) \not\sim_{\mu} (i+1,j-1) \\ q^{\sigma_{i-1,j+1}-\sigma_{ij}}q^{g_{ij}^+(\sigma)}, &\text{ if } (i-1,j+1) \not\sim_{\mu} (i,j) \sim_{\mu} (i+1,j-1)\\q^{\sigma_{i+1,j-1}+\sigma_{i-1,j+1}-\sigma_{ij}}q^{h_{ij}^+(\sigma)}, &\text{ if } (i-1,j+1) \not\sim_{\mu} (i,j) \not\sim_{\mu} (i+1,j-1)\end{cases}\end{split}\] 
To get the main result from the introduction we specialize to
\[h_{ij}^-(\pi)=f_{ij}^-(\pi)=f_{ij}^+(\pi)=g_{ij}^-(\pi)=0,\]
\[g_{ij}^+(\pi)=\pi_{i,j-1}-\pi_{i-1,j+1}-\beta_{i+1,j-1},\]
\[h_{ij}^+(\pi)=-\pi_{i-1,j+1},\]
and
\[m_{ij}(\pi)=\pi_{i,j-1}-\pi_{i+1,j-1}-\beta_{i+1,j-1}.\]
Define the set 
\[\mathrm{Hor}_{\mu}(\lambda):= \{(i,j) \in \mu : (i-1,j+1) \in \lambda/\mu\}\]
Write $\displaystyle A_{\lambda,\mu}(\sigma,z;q)=\sum_{\pi \in \Pi^{\lambda}_{\sigma}} \widetilde{W}_{\lambda,\mu}(\pi,z;q)$ with
\[H^{\lambda,\mu,\alpha}= G^{\lambda/\mu,\alpha}+V_{\lambda,\mu}\]
for \begin{equation} G^{\lambda/\mu,\alpha}= \displaystyle \sum_{(i,j) \in \lambda/\mu} \widetilde{b}_{ij}(\sigma,z;q) D_{\sigma_{ij}}\end{equation}
and
\begin{equation}\label{final6}\begin{split}
    V_{\lambda,\mu}(\sigma,z;q) &= \sum_{\substack{(i,j) \in C(\mu) \backslash \mathrm{Hor}_{\mu}(\lambda)}} z_{i,\ell(\lambda)}z_{i-1,i+j}(1-q^{\sigma_{i,j+1}})(1-q^{\sigma_{i+1,j}})\\
    &+\sum_{\substack{(i,j) \in C(\mu) \cap \mathrm{Hor}_{\mu}(\lambda)}} z_{i,\ell(\lambda)}z_{i-1,i+j}q^{-\pi_{i-1,j+1}}(1-q^{\sigma_{i,j+1}})(1-q^{\sigma_{i+1,j}})\\
    &+\sum_{\substack{1 \leq i \leq \ell(\mu)\\ (i,\mu_i) \not\in \mathrm{Hor}_{\mu}(\lambda)}} z_{i-1,\ell(\lambda)} (1-z_{i,i+\mu_i})(1-q^{\sigma_{i,\mu_i+1} })\\
    &+ \sum_{\substack{1 \leq i \leq \ell(\mu) \\(i,\mu_i) \in \mathrm{Hor}_{\mu}(\lambda)}}z_{i-1,\ell(\lambda)}q^{-\sigma_{i-1,\mu_i+1}}(1-z_{i,i+\mu_i})(1-q^{\sigma_{i,\mu_i+1}})
\end{split}
\end{equation}
Lastly, the actual Markov process on $\Pi^{\lambda,\alpha}$ here is the one with generator
\begin{equation}
    G^{\lambda,\alpha}_{\mu}=\sum_{(i,j) \in \lambda} \widetilde{b}_{ij}(\pi,z;q)D_{\pi_{ij}}.
\end{equation}
We can now restate Theorem \ref{bigtheorem}:
\begin{theorem}\label{bigresult}
    Fix a Young diagram $\lambda$. Suppose $\pi(t), t\geq 0$ is a Markov process on $\Pi^{\lambda,\alpha}$ with generator $G^{\lambda,\alpha}_{\mu}$ and initial law $K_{\sigma}^{\lambda,\mu}$ for some $\mu \subset \lambda^{\circ}$ with $\sigma \in \Pi^{\lambda/\mu,\alpha}$. Then $\sigma(t) = \pi(t)|_{\lambda/\mu}$ is a Markov process on $\Pi^{\lambda/\mu,\alpha}$ with generator 
    \begin{equation}\label{anotherok}
        L^{\lambda/\mu,\alpha}=A_{\lambda,\mu}^{-1} \circ H^{\lambda/\mu,\alpha} \circ A_{\lambda,\mu}. 
    \end{equation}
    For $t > 0$, the conditional law of $\pi(t)$ given $\{\sigma(s), s \leq t\}$ is $K^{\lambda,\mu}_{\sigma(t)}$.
\end{theorem}
We state a lemma which generalizes certain facts from Subsection $5.1$,
\begin{lemma}\label{genlemma1}
The following are true:
\begin{enumerate}
    \item For $\pi \in \Pi^{\lambda}_{\sigma}$ and $(i,j) \in \lambda/\mu$,
    \begin{equation}
        \widetilde{b}_{ij}(\sigma,z;q)L_{\sigma_{ij}}\widetilde{W}_{\mu}(\pi,z;q) = \widetilde{b}_{ij}(\pi,z;q)
    \end{equation}
    \item For $(i,j) \in \mu$,
    \[\begin{split} R_{\pi_{ij}}[\widetilde{b}_{ij}\widetilde{W}_{\lambda,\mu}]=
\widetilde{W}_{\lambda,\mu}z_{i-1,\ell(\lambda)}(&1-q^{\pi_{i,j+1}-\pi_{ij}})(1-z_{i,i+j}q^{\pi_{i+1,j}-\pi_{ij}})\times\\&\begin{cases} q^{\pi_{ij}-\pi_{i-1,j+1}}, &\text{ if } (i-1,j+1) \sim_{\mu} (i,j) \sim_{\mu} (i+1,j-1)\\
q^{\pi_{ij}-\pi_{i-1,j+1}}q^{R_{\pi_{ij}}f_{ij}^-(\pi)}, &\text{ if }(i-1,j+1) \sim_{\mu} (i,j) \not\sim_{\mu} (i+1,j-1)\\
q^{\pi_{ij}-\pi_{i-1,j+1}}q^{R_{\pi_{ij}}g_{ij}^-(\pi)}, &\text{ if }(i-1,j+1) \not\sim_{\mu} (i,j) \sim_{\mu} (i+1,j-1)\\
q^{\pi_{ij}-\pi_{i-1,j+1}}q^{R_{\pi_{ij}}h_{ij}^-(\pi)}, &\text{ if }(i-1,j+1) \not\sim_{\mu} (i,j) \not\sim_{\mu} (i+1,j-1)
\end{cases}        \end{split}\]
\end{enumerate}
\end{lemma}
\begin{proof}
    We begin with Statement $1$. Recall that
    \[D_{\lambda,\mu}(\pi,z;q) = \sum_{(i,j) \in \mu} \pi_{ij}^2 - \sum_{\substack{(i,j) \in \mu \\ (i+1,j-1) \in \mu}} \pi_{ij}\pi_{i+1,j-1} - \sum_{\substack{(i,j) \in \lambda/\mu \\ (i+1,j-1) \in \mu}} \pi_{ij}\pi_{i+1,j-1}-\sum_{\substack{(i,j) \in \mu\\(i+1,j-1) \in \lambda/\mu}} \pi_{ij}\pi_{i+1,j-1} \]
    If $(i,j) \in \lambda/\mu$ then there are four cases:
    \begin{enumerate}
        \item If $(i+1,j-1) \in \mu, (i-1,j+1) \in \mu$ then
        \[L_{\sigma_{ij}}D_{\lambda,\mu}(\pi,z;q)-D_{\lambda,\mu}(\pi,z;q) = \pi_{i+1,j-1}-\pi_{i-1,j+1}.\]
        \item If $(i+1,j-1) \in \lambda/\mu, (i-1,j+1) \in \mu$ then
        \[L_{\sigma_{ij}}D_{\lambda,\mu}(\pi,z;q)-D_{\lambda,\mu}(\pi,z;q)=\pi_{i-1,j+1}\]
        \item If $(i+1,j-1) \in \mu, (i-1,j+1) \in \lambda/\mu$ then
        \[L_{\sigma_{ij}}D_{\lambda,\mu}(\pi,z;q)-D_{\lambda,\mu}(\pi,z;q)=\pi_{i+1,j-1}\]
        \item If $(i+1,j-1) \in \lambda/\mu, (i-1,j+1) \in \lambda/\mu$ then
        \[L_{\sigma_{ij}}D_{\lambda,\mu}(\pi,z;q)- D_{\lambda,\mu}(\pi,z;q) = 0\]
    \end{enumerate}
    Putting all of the above cases together we see that the following holds for $(i,j) \in \lambda/\mu$,
    \[L_{\sigma_{ij}}D_{\lambda,\mu}(\pi,z;q)-D_{\lambda,\mu}(\pi,z;q)=\pi_{i+1,j-1}-\sigma_{i+1,j-1}+\pi_{i-1,j+1}-\sigma_{i-1,j+1}.\]
    \indent We now move on to Statement $2$. By performing similar casework as above, we can see that for $(i,j) \in \mu$ we have
    \[R_{\sigma_{ij}}D_{\lambda,\mu}(\pi,z;q)-D_{\lambda,\mu}=2\pi_{ij}-\pi_{i+1,j-1}-\pi_{i-1,j+1}+1.\]
\end{proof}
\indent We now proceed with the proof of our main theorem,
\begin{proof}[Proof of Theorem \ref{bigresult}]
    Similar to Theorem \ref{bigdriftlessresult}, we can first remark that Lemma \ref{genlemma1} implies
    \[(H^{\lambda/\mu,\alpha}\Lambda F)(\sigma)=\sum_{\pi \in \Pi^{\lambda}_{\sigma}} \bigg([H^{\lambda/\mu,\alpha} \widetilde{W}_{\lambda,\mu}(\pi,z;q)]F(\pi) + \widetilde{W}_{\lambda,\mu}(\pi,z;q) \sum_{(i,j) \in \lambda/\mu} \widetilde{b}_{ij}(\pi,z;q) D_{\sigma_{ij}}F(\pi)\bigg).\]
    Thus the statement $H^{\lambda/\mu,\alpha} \circ \Lambda=\Lambda \circ G^{\lambda,\alpha}$ reduces to proving \[H^{\lambda/\mu,\alpha} \widetilde{W}_{\lambda,\mu} = z_{\ell(\mu),\ell(\lambda)} \bigg(\sum_{(i,j) \in \mu} \widetilde{b}_{ij}(\pi,z;q)D_{\pi_{ij}}\bigg)' \widetilde{W}_{\lambda,\mu}\] where $(-)'$ denotes the same formal adjoint we used in Theorem \ref{bigdriftlessresult}.\newline
    \indent Due to statement $2$ of Lemma \ref{genlemma1}, and $f_{ij}^-=g_{ij}^-=h_{ij}^-=0$, we set $b_{ij}'(\pi;z,q) = z_{i-1,\ell(\lambda)}q^{\pi_{ij}-\pi_{i-1,j+1}}(1-q^{\pi_{i,j+1-\pi_{ij}}})(1-z_{i,i+j}q^{\pi_{i+1,j}-\pi_{ij}})$ for all $(i,j) \in \mathbb{N}^2$. Once again, the theorem reduces to proving 
    \begin{equation}\label{ok31}\sum_{(i,j) \in \lambda/\mu} [\widetilde{b}_{ij}(\pi,z;q) - \widetilde{b}_{ij}(\sigma,z;q)] + V_{\lambda,\mu}(\sigma,z;q) = \sum_{(i,j) \in \mu} [b_{ij}'(\pi,z;q) - \widetilde{b}_{ij}(\pi,z;q)].\end{equation}
    In Theorem \ref{bigdriftlessresult} we introduced a Young diagram $\hat{\mu}$ by adding an $(r+1,r+1)$-hook to $\mu$. Here we will use a more delicate construction of $\mu$, which applies to non-staircase $\mu$, which will aid us in proving the theorem. Let $\mu^{(0)}=\mu$, 
    \begin{equation}
        \mu^{(1)}=\mu \cup \{(i,j) \in \lambda: (i,j) \not\in \mu, (i+1,j-1) \in \mu\} \cup \{(i,j) \in \lambda: (i,j) \not\in \mu, (i-1,j+1) \in \mu\} \cup \{(0,0)\},
    \end{equation}
    and for every $k \geq 1$ set 
    \begin{equation}
        \mu^{(k+1)}:=\mu^{(k)}\cup\{(i,j) \not\in\mu^{(k)}: (i+1,j-1) \in \mu^{(k)}\} \cup \{(i,j) \not\in\mu^{(k)}: (i-1,j+1) \in \mu^{(k)}\}.
    \end{equation}
    At every $k \geq 0$, $\mu^{(k+1)}$ adds hook shapes to $\mu^{(k)}$. Since $\mu$ is finite, there exists finite $k_0>0$ for which $\mu^{(k_0)}=\mu^{(k)}$ for all $k \geq k_0$, and $\mu^{(k_0)}$ is a staircase shape. Let $\hat{\mu}$ denote the largest staircase shape containing both $\mu^{(k_0)}$ and $\lambda$. We take the convention $\mu^{(k)}=\mu$ for $k \leq 0$. We state a lemma which we prove in Appendix $A$:
    \begin{lemma}\label{genlemma2} There exist $T_{\lambda,\mu}(\pi,z;q)$ and $U_{\lambda,\mu}(\pi,z;q)$
       such that
            \begin{equation}
            T_{\lambda,\mu}(\pi,z;q) = \sum_{(i,j) \in \hat{\mu}} [b_{ij}'(\pi,z;q) - \widetilde{b}_{ij}(\pi,z;q)],
        \end{equation}
        \begin{equation}
            \sum_{(i,j) \in \mathbb{N}^2/\mu}[\widetilde{b}_{ij}(\pi,z;q) - \widetilde{b}_{ij}(\sigma,z;q)]+U_{\lambda,\mu}(\pi,z;q) = -\sum_{(i,j) \in \hat{\mu}/\mu} [b_{ij}'(\pi,z;q) - \widetilde{b}_{ij}(\pi,z;q)],
        \end{equation}
    and $T_{\lambda,\mu}(\pi,z;q) + U_{\lambda,\mu}(\pi,z;q) = V_{\lambda,\mu}(\sigma,z;q)$ hold.
    \end{lemma}
    Observe that
    \[\sum_{(i,j) \in \mathbb{N}^2\backslash \mu}[\widetilde{b}_{ij}(\pi,z;q)-\widetilde{b}_{ij}(\sigma,z;q)]= \sum_{(i,j) \in \lambda/\mu} [\widetilde{b}_{ij}(\pi,z;q)-\widetilde{b}_{ij}(\sigma,z;q)]\]
    and
    \[\bigg(\sum_{(i,j) \in \tilde{\mu}}-\sum_{(i,j) \in \tilde{\mu}/\mu}\bigg)[b_{ij}'(\pi,z;q)-\widetilde{b}_{ij}(\pi,z;q)]=\sum_{(i,j) \in \mu}[b_{ij}'(\pi,z;q)-\widetilde{b}_{ij}(\pi,z;q)],\]
    and so
    \begin{equation}
        \sum_{(i,j) \in\lambda/\mu} [\widetilde{b}_{ij}(\pi,z;q)-\widetilde{b}_{ij}(\sigma,z;q)]+(T_{\lambda,\mu}(\pi,z;q)+U_{\lambda,\mu}(\pi,z;q))=\sum_{(i,j) \in \mu}[b_{ij}'(\pi,z;q)-\widetilde{b}_{ij}(\pi,z;q)].
    \end{equation}
    Since $T_{\lambda,\mu}(\pi,z;q) + U_{\lambda,\mu}(\pi,z;q) = V_{\lambda,\mu}(\sigma,z;q)$ we have completed our proof that $H^{\lambda/\mu,\alpha}\circ \Lambda = \Lambda \circ G^{\lambda,\alpha}$.
\end{proof}
\begin{appendices}
    \section{Proof of Lemma $5.6$}
    \indent To prove this lemma we must consider the hook shapes embedded in $\lambda$. We refer to $F \subset \lambda$ a hook with origin $(i,j) \in \lambda$, vertical length $\ell$, and horizontal length $\omega$ if $F=\{(i+a-1,j+b-1) : (a,b) \in (\{1\}\times\{1,\dots,\ell\}) \cup (\{1,\dots,\omega\} \times \{1\})\}$. Note that the Young diagram $F=(\ell,1^{\omega-1})$ is a hook with origin $(1,1)$, vertical length $\ell$, and horizontal length $\omega$. A hook is called degenerate if $\ell=1$ or $\omega=1$, and nondegenerate otherwise.
\begin{proof}[Proof of Lemma \ref{genlemma2}]
    As a disclaimer, all of our computations involving hooks will first deal with nondegenerate hooks. This is not because of any inherent technical difficulty, it is only to keep the indexing manageable. Once we complete our calculations for nondegenerate hooks we will make a brief remark on how the same calculations hold for degenerate hooks, with only slight differences in indexing appearing. \newline
    \indent For $k \geq 0$, every $(i,j) \in \mu^{(k+1)}/\mu^{(k)}$ belongs to a unique hook shape $F_{ij} \subset \mu^{(k+1)}/\mu^{(k)}$ which we will describe in relation to another hook shape $G_{ij} \subset \mu^{(k)}$. Due to $(i,j) \in \mu^{(k+1)}/\mu^{(k)}$ we must have $(i-1,j-1) \in \mu^{(k)}$, let $G_{ij}$ be the largest hook in $\mu^{(k)}$ which contains the point $(i-1,j-1)$. Denote the vertical length of $G_{ij}$ by $\ell_{ij}+1$ and the horizontal length of $G_{ij}$ by $\omega_{ij}+1$. Define $F_{ij}$ to be the hook containing $(i,j)$ with vertical length $\ell_{ij}$ and horizontal length $\omega_{ij}$. \newline
    \indent It is useful to think of the hooks $F_{ij}$ and $G_{ij}$ as being embedded in larger hooks $\hat{F}_{ij},\hat{G}_{ij} \subset \hat{\mu}$. Let $\hat{F}_{ij}$ be the largest hook in $\hat{\mu}$ which contains $F_{ij}$, and let $\hat{G}_{ij}$ be the largest hook in $\hat{\mu}$ which contains $G_{ij}$. We will denote the vertical length of $\hat{F}_{ij}$ by $\hat{\ell}_{ij}$ and the horizontal length of $\hat{F}_{ij}$ by $\hat{\omega}_{ij}$, noting that the staircase shape of $\hat{\mu}$ guarantees that $\hat{\ell}_{ij}=\hat{\omega}_{ij}$.\newline
    \indent In order to decompose $\hat{\mu}/\mu^{\circ}$ into hook shapes and pair these hooks together, it will be convenient to define the set $\mathcal{O}_1$ of $(i,j) \in \mu^{(1)}/\mu$ for which $(i,j)$ is the origin of the hook $F_{ij}$. We may classify points in $\mathcal{O}_1$ by how their hooks $G_{ij} \subset \mu$ glue together with other hooks $G_{i',j'} \subset \mu$ for $(i',j')\neq (i,j)$. Specifically, we let
    \[\mathcal{O}_1 = \mathcal{O}_1^{0	} \cup \mathcal{O}_1^{+} \cup \mathcal{O}_1^{-} \cup \mathcal{O}_1^{\bullet} \]
    where $(i,j) \in \mathcal{O}_1^0$ when $(i-2,j+\ell_{ij}+1),(i+\omega_{ij}+1,j-2) \in \mu$, we say $(i,j) \in \mathcal{O}_1^+$ if $(i-2,j+\ell_{ij}+1)\in \lambda/\mu,(i+\omega_{ij}+1,j-2) \in \mu$, we say $(i,j) \in \mathcal{O}_1^-$ if $(i-2,j+\ell_{ij}+1) \in \mu,(i+\omega_{ij}+1,j-2) \in \lambda/\mu$, and finally $(i,j) \in \mathcal{O}_1^{\bullet}$ if $(i-2,j+\ell_{ij}+1),(i+\omega_{ij}+1,j-2) \in \lambda/\mu$. Figure \ref{fig:fig4} depicts the different classes of hooks we are dealing with. \begin{figure}[h]
    \centering
    \captionsetup{justification=centering}
    \includegraphics[scale=0.4]{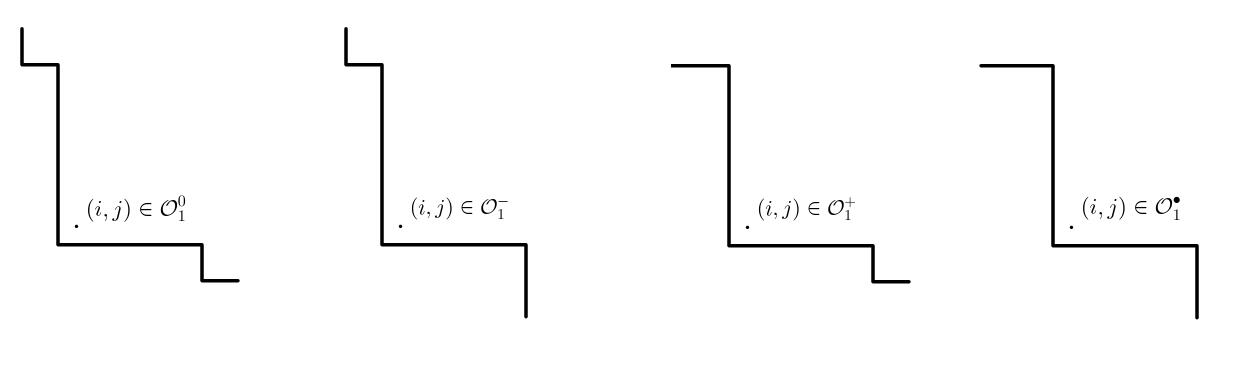}
    \caption{The different possible $G_{ij}$ shapes for $(i,j) \in \mathcal{O}_1$.}
    \label{fig:fig4}
\end{figure}\newline
    \indent For a hook $F$ with origin $(i,j)$, vertical length $\ell$, and horizontal length $\omega$ we let $\mathscr{S}(F)$ denote the staircase generated by $F$ and define it as the set $\mathscr{S}(F)=\{(a,b) \in \llbracket i,i+\omega_{ij}-1\rrbracket \times \llbracket j,j+\ell-1\rrbracket : a+b \leq i+j+\ell-1\}$. Let $C(\mu)=\{(i,j)\in\mu:(i+1,j)\not\in\mu,(i,j+1)\not\in\mu\}$ be the set of external corners and
    \[\widetilde{C}(\mu)=\{(i,j) \in C(\mu): (i,j) \not\in \hat{G}_{i',j'} \text{ for all } (i',j') \in \mathcal{O}_{1}\}\]
    the set of \emph{exceptional} external corners. For each $(i,j) \in \widetilde{C}(\mu)$ we set $\widetilde{F}_{ij}=\{(i+1,j)\}$ if $(i-1,j+1),(i+1,j-1) \in \widetilde{C}(\mu)$, and $\widetilde{F}_{ij}=\{(i,j+1),(i+1,j)\}$ if $(i-1,j+1)\not\in \widetilde{C}(\mu)$ or $(i+1,j-1)\not\in \widetilde{C}(\mu)$. Let $S_{ij}=\{(a,b) \in \hat{\mu}: a=i \text{ or } b=j\}\backslash \{(i,j)\}$ and we take the convention that $\widetilde{F}_{ij}=\emptyset$ for $(i,j)\not\in \widetilde{C}(\mu)$, we also define the regions $\displaystyle \mathscr{S}_b=\bigcup_{(i,j) \in \mathcal{O}_1} \mathscr{S}(\hat{F}_{ij})$ and $\displaystyle \mathscr{S}_e=\bigcup_{(i,j) \in \widetilde{C}(\mu)} S_{ij}$. This allows us to write
    \[\begin{split} -U_{\lambda,\mu}(\pi,z;q) &=\sum_{v \in \mathscr{S}_b \cup \mathscr{S}_e} [b_v'(\pi,z;q)-\widetilde{b}_v(\sigma,z;q)]\\&= \sum_{v \in \mathscr{S}_b \cup \mathscr{S}_e} [\hat{b}_v'(\pi,z;q) - \hat{b}_v(\pi,z;q)] + \sum_{v \in \bigcup_{(i,j) \in \mathcal{O}_1} \hat{F}_{ij} \cup \bigcup_{(i,j) \in \widetilde{C}(\mu)}\widetilde{F}_{ij}} [\widetilde{b}_v(\pi,z;q)-\widetilde{b}_v(\sigma,z;q)] \\
    &+\sum_{v \in \mathscr{S}_b \cup \mathscr{S}_e}\bigg([b_v'(\pi,z;q)-\widetilde{b}_v(\pi,z;q)]-[\hat{b}_v'(\pi,z;q)-\hat{b}_v(\pi,z;q)]\bigg)
    \end{split}\]
    where $\hat{b}_{ij}(\pi,z;q):=z_{i,\ell(\lambda)}q^{\pi_{i+1,j-1}-\pi_{ij}}(1-q^{\pi_{ij}-\pi_{i,j-1}})(1-z_{i-1,i+j-1}q^{\pi_{ij}-\pi_{i-1,j}})$ and \[\hat{b}'_{ij}(\pi,z;q):=z_{i-1,\ell(\lambda)}q^{\pi_{ij}-\pi_{i-1,j+1}}(1-q^{\pi_{i,j+1}-\pi_{ij}})(1-z_{i,i+j}q^{\pi_{i+1,j}-\pi_{ij}}).\] Calculations analogous to the ones we did in Subsection $5.1$ give the values
    \begin{equation}\begin{split}\label{generalstaircasestuff}
        &\sum_{v \in \mathscr{S}(\hat{G}_{ij})} [\hat{b}_v'(\pi,z;q)-\hat{b}_v(\pi,z;q)]=\\& \sum_{\substack{v=(a,b) \in \mathscr{S}(\hat{G}_{ij})\\ a+b = i+j+\hat{\ell}_{ij}-1} } [z_{a-1,\ell(\lambda)}z_{a,a+b}q^{\pi_{a+1,b}+\pi_{a,b+1}-\pi_{ab}-\pi_{a-1,b+1}}-z_{a-1,\ell(\lambda)}q^{\pi_{a,b+1}-\pi_{a-1,b+1}}-z_{a-1,\ell(\lambda)}z_{a,a+b} q^{\pi_{a+1,b}-\pi_{a-1,b+1}}]\\
        &+\sum_{a=i-1}^{i+\hat{\omega}_{ij}}z_{a,\ell(\lambda)}[z_{a-1,a+j-2}q^{\pi_{a+1,j-2}-\pi_{a-1,j-1}}+q^{\pi_{a+1,j-1}-\pi_{a,j-1}}\mathbbm{1}_{a \neq i+\hat{\omega}_{ij}}+q^{\pi_{a+1,j-2}-\pi_{a,j-2}}-q^{\pi_{a+1,j-2}-\pi_{a,j-1}}\\&-z_{a-1,a+j-2}q^{\pi_{a,j-1}-\pi_{a,j-2}+\pi_{a+1,j-2}-\pi_{a-1,j-1}}]
        +\sum_{b=j-1}^{j+\hat{\ell}_{ij}} z_{i-2,\ell(\lambda)}[q^{\pi_{i-1,b}-\pi_{i-2,b+1}}-q^{\pi_{i-1,b+1}-\pi_{i-2,b+1}}\mathbbm{1}_{b \neq j+\hat{\ell}_{ij}}],
    \end{split}\end{equation}
        \[\begin{split}&\sum_{(i,j) \in \hat{\mu}}[\hat{b}_{ij}'(\pi,z;q)-\hat{b}_{ij}(\pi,z;q)]=\sum_{a} z_{a,\ell(\lambda)}+\\&\sum_{\substack{(i,j) \in \hat{\mu}\\(i,j+1) \not\in \hat{\mu} \text{ or }(i+1,j) \not\in \hat{\mu}}} [z_{i-1,\ell(\lambda)}z_{i,i+j}q^{\pi_{i+1,j}+\pi_{i,j+1}-\pi_{ij}-\pi_{i-1,j+1}}-z_{i-1,\ell(\lambda)}q^{\pi_{i,j+1}-\pi_{i-1,j+1}}-z_{i-1,\ell(\lambda)}z_{i,i+j} q^{\pi_{i+1,j}-\pi_{i-1,j+1}}]
    \end{split}\]
    for the above telescoping sums. We assume that $(i,j) \in \mathcal{O}_1$ is a nondegenerate hook until we say otherwise. We will make further remarks about this later in this appendix. To understand $T_{\lambda,\mu}$ we can start by writing
    \[\begin{split}
        T_{\lambda,\mu} &= \sum_{(i,j) \in \hat{\mu}}[\hat{b}_{ij}'(\pi,z;q)-\hat{b}_{ij}(\pi,z;q)]+\sum_{(i,j) \in \hat{\mu}}\bigg([b'_{ij}(\pi,z;q)-\widetilde{b}_{ij}(\pi,z;q)]-[\hat{b}_{ij}'(\pi,z;q)-\hat{b}_{ij}(\pi,z;q)]\bigg)\\
        &=\sum_{(i,j) \in \hat{\mu}}[\hat{b}_{ij}'(\pi,z;q)-\hat{b}_{ij}(\pi,z;q)]+\sum_{(a,b) \in \mathscr{S}_b \cup \mathscr{S}_e}\bigg([b'_{ab}(\pi,z;q)-\widetilde{b}_{ab}(\pi,z;q)]-[\hat{b}_{ab}'(\pi,z;q)-\hat{b}_{ab}(\pi,z;q)]\bigg)+\\
        &\sum_{(a,b) \in \widetilde{C}(\mu)\cup \bigcup_{(i,j) \in \mathcal{O}_1}G_{ij}}\bigg([b'_{ab}(\pi,z;q)-\widetilde{b}_{ab}(\pi,z;q)]-[\hat{b}_{ab}'(\pi,z;q)-\hat{b}_{ab}(\pi,z;q)]\bigg).\end{split}\]
    From the above formulas for $U_{\lambda,\mu}$ and $T_{\lambda,\mu}$ we may deduce
    \begin{equation}\label{final5}\begin{split}
        &U_{\lambda,\mu}(\pi,z;q) + T_{\lambda,\mu}(\pi,z;q)=\sum_{a=-1}^{\hat{\omega}_{ij}} z_{a,\ell(\lambda)}-\sum_{v \in \mathscr{S}_b \cup \mathscr{S}_e\cup \widetilde{C}(\mu) \cup \bigcup_{(i,j) \in \mathcal{O}_1}G_{ij}}[\hat{b}_v'(\pi,z;q) - \hat{b}_v(\pi,z;q)]-\\&\sum_{v \in \bigcup_{(i,j) \in \mathcal{O}_1} \hat{F}_{ij} \cup \bigcup_{(i,j) \in \widetilde{C}(\mu)} \widetilde{F}_{ij}} [\widetilde{b}_v(\pi,z;q) - \widetilde{b}_v(\sigma,z;q)] +\sum_{(a,b) \in \widetilde{C}(\mu)\cup \bigcup_{(i,j) \in \mathcal{O}_1}G_{ij}}[b'_{ab}(\pi,z;q)-\widetilde{b}_{ab}(\pi,z;q)]+\\&\sum_{\substack{(i,j) \in \hat{\mu}\\(i,j+1) \not\in \hat{\mu} \text{ or }(i+1,j) \not\in \hat{\mu}}} [z_{i-1,\ell(\lambda)}z_{i,i+j}q^{\pi_{i+1,j}+\pi_{i,j+1}-\pi_{ij}-\pi_{i-1,j+1}}-z_{i-1,\ell(\lambda)}q^{\pi_{i,j+1}-\pi_{i-1,j+1}}-z_{i-1,\ell(\lambda)}z_{i,i+j} q^{\pi_{i+1,j}-\pi_{i-1,j+1}}].
    \end{split}\end{equation}
\begin{center}
\renewcommand{\arraystretch}{1.45}
\setlength{\tabcolsep}{0.6em}
\begin{table}
\begin{tabular}{|>{\centering\arraybackslash}m{18em}|m{30em}|}

\hline
\textbf{$U_{\lambda,\mu},\quad \quad\quad(i,j)\in\mathcal{O}_1$} & \\[2pt]
\hline

$\displaystyle 
\sum_{\substack{w=(a,b)\in F_{ij}\\ a\ge i+1}}
\left[\widetilde{b}_w(\pi,z;q)-\widetilde{b}_w(\sigma,z;q)\right]
$
&
$\displaystyle 
\sum_{a=i+1}^{i+\omega_{ij}-1}
z_{a,\ell(\lambda)} q^{-\pi_{aj}} q^{f_{aj}^+(\pi)}
(1-z_{a-1,a+j-1}q^{\pi_{aj}-\pi_{a-1,j}})
[
q^{\pi_{a+1,j-1}}(1-q^{\pi_{aj}-\pi_{a,j-1}})
-
q^{f_{aj}^+(\sigma)-f_{aj}^+(\pi)}(1-q^{\pi_{aj}})
]
$
\\
\hline

$\displaystyle 
\sum_{\substack{w=(a,b)\in F_{ij}\\ b\ge j+1}}
[\widetilde{b}_w(\pi,z;q)-\widetilde{b}_w(\sigma,z;q)]
$
&
$\displaystyle 
\sum_{b=j+1}^{j+\ell_{ij}-1}
z_{i,\ell(\lambda)} q^{-\pi_{ib}} q^{g_{ib}^+(\pi)}
(1-q^{\pi_{ib}-\pi_{i,b-1}})
[
q^{\pi_{i-1,b+1}}(1-z_{i-1,i+b-1}q^{\pi_{ib}-\pi_{i-1,b}})
-
q^{g_{ib}^+(\sigma)-g_{ib}^+(\pi)}
(1-z_{i-1,i+b-1} q^{\pi_{ib}})
]
$
\\
\hline

$\widetilde{b}_{ij}(\pi,z;q)-\widetilde{b}_{ij}(\sigma,z;q)$
&
$\displaystyle 
q^{-\pi_{ij}} q^{h_{ij}^+(\pi)}
[
q^{\pi_{i+1,j-1}+\pi_{i-1,j+1}}
(1-q^{\pi_{ij}-\pi_{i,j-1}})
(1-z_{i-1,i+j-1}q^{\pi_{ij}-\pi_{i-1,j}})
-
q^{h_{ij}^+(\sigma)-h_{ij}^+(\pi)}
(1-q^{\pi_{ij}})
(1-z_{i-1,i+j-1}q^{\pi_{ij}})
]
$
\\
\hline

$\widetilde{b}_{i,j+\ell_{ij}}(\pi,z;q)-\widetilde{b}_{i,j+\ell_{ij}}(\sigma,z;q)$
&
$\displaystyle 
z_{i,\ell(\lambda)} q^{m_{ij}(\pi)}
q^{\pi_{i-1,j+1}-\pi_{ij}}
(1-q^{\pi_{i,j+\ell_{ij}}-\pi_{i,j+\ell_{ij}-1}})
[
(1-z_{i-1,i+j_{\ell_{ij}}-1}q^{\pi_{i,j+\ell_{ij}}-\pi_{i-1,j+\ell_{ij}}})
-
q^{m_{ij}(\sigma)-m_{ij}(\pi)}
(1-z_{i-1,i+j_{\ell_{ij}}-1}q^{\pi_{i,j+\ell_{ij}}})
]
$
\\
\hline

$\widetilde{b}_{i+\omega_{ij},j}(\pi,z;q)-\widetilde{b}_{i+\omega_{ij},j}(\sigma,z;q)$
&
$\displaystyle 
z_{i,\ell(\lambda)} 
q^{\pi_{i+\omega_{ij}+1,j-1}-\pi_{i+\omega_{ij},j}}
(1-z_{i+\omega_{ij}-1,i+\omega_{ij}+j-1}q^{\pi_{i+\omega_{ij},j}-\pi_{i+\omega_{ij}-1,j}})
[
(1-q^{\pi_{i+\omega_{ij},j}-\pi_{i+\omega_{ij},j-1}})
-
(1-q^{\pi_{i+\omega_{ij},j}})
]
$
\\
\hline

 $T_{\lambda,\mu},\quad\quad\quad(i,j) \in \mathcal{O}_1^0$&\\ \hline
 $\displaystyle \sum_{\substack{w=(a,b) \in G_{ij}\\ a \geq i+1 }} [b_w'(\pi,z;q)-\widetilde{b}_w(\pi,z;q)]$ & $\displaystyle \sum_{a=i+1}^{i+\omega_{ij}} q^{g_{a,j-1}^-(\pi)}[z_{a-1,\ell(\lambda)}q^{\pi_{a,j-1}-\pi_{a-1,j}}q^{B_{\pi_{a,j-1}}g_{a,j-1}^-(\pi)}(1-q^{\pi_{aj}-\pi_{a,j-1}})(1-z_{a,a+j-1}q^{\pi_{a+1,j-1}-\pi_{a,j-1}})-z_{a,\ell(\lambda)}q^{\pi_{a+1,j-2}-\pi_{a,j-1}}(1-q^{\pi_{a,j-1}-\pi_{a,j-2}})(1-z_{a-1,a+j-2}q^{\pi_{a,j-1}-\pi_{a-1,j-1}})]$\\
  \hline
  $\displaystyle \sum_{\substack{w=(a,b) \in G_{ij}\\a=i-1}} [b_w'(\pi,z;q)-\widetilde{b}_w(\pi,z;q)]$ & $\displaystyle \sum_{b=j-1}^{j+\ell_{ij}}q^{f_{i-1,b}^-(\pi)}[z_{i-2,\ell(\lambda)}q^{\pi_{i-1,b}-\pi_{i-2,b+1}}q^{B_{\pi_{i-1,b}}f_{i-1,b}^-(\pi)}(1-q^{\pi_{i-1,b+1}-\pi_{i-1,b}})(1-z_{i-1,i+b-1}q^{\pi_{ib}-\pi_{i-1,b}})-z_{i-1,\ell(\lambda)}q^{\pi_{i,b-1}-\pi_{i-1,b}}(1-q^{\pi_{i-1,b}-\pi_{i-1,b-2}})(1-z_{i-1,i+b-2}q^{\pi_{i-1,b}-\pi_{i-2,b}})]$\\
  \hline
  $\displaystyle b_{i,j-1}'(\pi,z;q)-\widetilde{b}_{i,j-1}(\pi,z;q)$&$\displaystyle z_{i-1,\ell(\lambda)}q^{\pi_{i,j-1}-\pi_{i-1,j}}(1-q^{\pi_{ij}-\pi_{i,j-1}})(1-z_{i,i+j-1}q^{\pi_{i+1,j-1}-\pi_{i,j-1}})-z_{i,\ell(\lambda)}q^{\pi_{i+1,j-2}-\pi_{i,j-1}}(1-q^{\pi_{i,j-1}-\pi_{i,j-2}})(1-z_{i-1,i+j-2})+z_{i-2,\ell(\lambda)}[q^{\pi_{i-1,j-1}-\pi_{i-2,j}}-q^{\pi_{i-1,j}-\pi_{i-2,j}}]$\\
    \hline
    $\displaystyle b_{i-1,j}'(\pi,z;q)-\widetilde{b}_{i-1,j}$&$\displaystyle -[z_{i-2,\ell(\lambda)}q^{\pi_{i-1,j}-\pi_{i-2,j+1}}(1-q^{\pi_{i-1,j+1}-\pi_{i-1,j}})(1-z_{i-1,i+j-1}q^{\pi_{ij}-\pi_{i-1,j}})-z_{i-1,\ell(\lambda)}q^{\pi_{i,j-1}-\pi_{i-1,j}}(1-q^{\pi_{i-1,j}-\pi_{i-1,j-1}})(1-z_{i-2,i+j-2}q^{\pi_{i-1,j}-\pi_{i-2,j}})] +z_{i-2,\ell(\lambda)}[q^{\pi_{i-1,j}-\pi_{i-2,j+1}}-q^{\pi_{i-1,j+1}-\pi_{i-2,j+1}}]$\\
        \hline
    $T_{\lambda,\mu},\quad\quad\quad(i,j) \in \mathcal{O}_1^+$ & \\
    \hline
    $b_{i-1,j+\ell_{ij}}'(\pi,z;q)-\widetilde{b}_{i-1,j+\ell_{ij}}(\pi,z;q)$ &$[z_{i-2,\ell(\lambda)}q^{\pi_{i-1,j+\ell_{ij}}-\pi_{i-2,j+\ell_{ij}+1}}q^{B_{\pi_{i-1,j+\ell_{ij}}}h_{i-1,j+\ell_{ij}}^-(\pi)}(1-q^{\pi_{i-1,j+\ell_{ij}+1}-\pi_{i-1,j+\ell_{ij}}})(1-z_{i-1,i+j+\ell_{ij}-1}q^{\pi_{i,j+\ell_{ij}}-\pi_{i-1,j+\ell_{ij}}})-z_{i-1,\ell(\lambda)}q^{\pi_{i,j+\ell_{ij}-1}-\pi_{i-1,j+\ell_{ij}}}(1-q^{\pi_{i-1,j+\ell_{ij}}-\pi_{i-1,j+\ell_{ij}-1}})(1-z_{i-2,i+j+\ell_{ij}-2}q^{\pi_{i-1,j+\ell_{ij}}-\pi_{i-2,j+\ell_{ij}}})]q^{h_{i-1,j+\ell_{ij}}^-(\pi)}$\\
    \hline
    $T_{\lambda,\mu},\quad\quad\quad(i,j) \in \mathcal{O}_1^-$ & \\
    \hline
    $b_{i+\omega_{ij},j-1}'(\pi,z;q) - \widetilde{b}_{i+\omega_{ij},j-1}(\pi,z;q)$ & $[z_{i+\omega_{ij}-1,\ell(\lambda)}q^{B_{\pi_{i+\omega_{ij},j-1}}h_{i+\omega_{ij},j-1}^-(\pi)}q^{\pi_{i+\omega_{ij},j-1}-\pi_{i+\omega_{ij}-1,j}}(1-q^{\pi_{i+\omega_{ij},j}-\pi_{i+\omega_{ij},j-1}})(1-z_{i+\omega_{ij},i+\omega_{ij}+j-1}q^{\pi_{i+\omega_{ij}+1,j-1}-\pi_{i+\omega_{ij},j-1}})- z_{i+\omega_{ij},\ell(\lambda)}q^{\pi_{i+\omega_{ij}+1,j-2}-\pi_{i+\omega_{ij},j-1}}(1-q^{\pi_{i+\omega_{ij},j-1}-\pi_{i+\omega_{ij},j-2}})(1-z_{i+\omega_{ij}-1,i+\omega_{ij}+j-2}q^{\pi_{i+\omega_{ij},j-1}-\pi_{i+\omega_{ij}-1,j-1}})]q^{h_{i+\omega_{ij},j-1}^-(\pi)}$ \\
    \hline

\end{tabular}
\caption{Contributions to $U_{\lambda,\mu}$ and $T_{\lambda,\mu}$ coming from the hook at $(i,j) \in \mathcal{O}_1$.}
\end{table}
\end{center}
    \indent Combining terms from the second-to-last sum in \eqref{generalstaircasestuff} with the first row under $U_{\lambda,\mu},(i,j)\in\mathcal{O}_1$ in the above table and the first row under $T_{\lambda,\mu}, (i,j) \in \mathcal{O}_1^0$ in the above table defines the quantity
        \[\begin{split}
        &\Phi_{ij}^{(a)}(\pi,z;q):=z_{a,\ell(\lambda)}[z_{a-1,a+j-2}q^{\pi_{a+1,j-2}-\pi_{a-1,j-1}}+q^{\pi_{a+1,j-1}-\pi_{a,j-1}}+q^{\pi_{a+1,j-2}-\pi_{a,j-2}}-q^{\pi_{a+1,j-2}-\pi_{a,j-1}}\\
        &-z_{a-1,a+j-2}q^{\pi_{a,j-1}-\pi_{a,j-2}+\pi_{a+1,j-2}-\pi_{a-1,j-1}}]+z_{a,\ell(\lambda)}q^{-\pi_{aj}}q^{f_{aj}^+(\pi)}(1-z_{a-1,a+j-1}q^{\pi_{aj}-\pi_{a-1,j}})[q^{\pi_{a+1,j-1}}(1-q^{\pi_{aj}-\pi_{a,j-1}})\\&-q^{f_{aj}^+(\sigma)-f_{aj}^+(\pi)}(1-q^{\pi_{aj}})]-q^{g_{a,j-1}^-(\pi)}[z_{a-1,\ell(\lambda)}q^{\pi_{a,j-1}-\pi_{a-1,j}}q^{B_{\pi_{a,j-1}}g_{a,j-1}^-(\pi)}(1-q^{\pi_{aj}-\pi_{a,j-1}})(1-z_{a,a+j-1}q^{\pi_{a+1,j-1}-\pi_{a,j-1}})\\&-z_{a,\ell(\lambda)}q^{\pi_{a+1,j-2}-\pi_{a,j-1}}(1-q^{\pi_{a,j-1}-\pi_{a,j-2}})(1-z_{a-1,a+j-2}q^{\pi_{a,j-1}-\pi_{a-1,j-1}})]\\
        &=(z_{a,\ell(\lambda)}q^{\pi_{a+1,j-1}-\pi_{aj}}-z_{a-1,\ell(\lambda)}q^{\pi_{a,j-1}-\pi_{a-1,j}})+z_{a,\ell(\lambda)}[1-q^{-\pi_{aj}}+z_{a-1,a+j-1}q^{-\pi_{a-1,j}}]+z_{a-1,\ell(\lambda)}(1-z_{a,a+j-1})q^{\pi_{aj}-\pi_{a-1,j}}
    \end{split}\]
    for $a=i+1,\dots,i+\omega_{ij}-1$. Combining terms from the last sum in \eqref{generalstaircasestuff} with the second row under $U_{\lambda,\mu},(i,j)\in\mathcal{O}_1$ in the above table and the second row under $T_{\lambda,\mu}, (i,j) \in \mathcal{O}_1^0$ in the above table defines the quantity
    \[\begin{split}
        &\Psi_{ij}^{(b)}(\pi,z;q):=
        z_{i,\ell(\lambda)}q^{-\pi_{ib}}q^{g_{ib}^+(\pi)}(1-q^{\pi_{ib}-\pi_{i,b-1}})[q^{\pi_{i-1,b+1}}(1-z_{i-1,i+b-1}q^{\pi_{ib}-\pi_{i-1,b}})\\&-q^{f_{ij}^-(\pi)}[z_{i-2,\ell(\lambda)}q^{\pi_{i-1,b}-\pi_{i-2,b+1}}q^{B_{\pi_{ij}}f_{ij}^-(\pi)}(1-q^{\pi_{i-1,b+1}-\pi_{i-1,b}})(1-z_{i-1,i+b-1}q^{\pi_{ib}-\pi_{i-1,b}})\\&-z_{i-1,\ell(\lambda)}q^{\pi_{i,b-1}-\pi_{i-1,b}}(1-q^{\pi_{i-1,b}-\pi_{i-1,b-1}})(1-z_{i-1,i+b-2}q^{\pi_{i-1,b}-\pi_{i-2,b}})]\\
        &=-(z_{i-2,\ell(\lambda)}z_{i-1,i+b-1}q^{\pi_{ib}+\pi_{i-1,b+1}-\pi_{i-1,b}-\pi_{i-2,b+1}}-z_{i-2,i+b-2}z_{i-1,\ell(\lambda)}q^{\pi_{i-1,b}-\pi_{i-1,b-1}+\pi_{i,b-1}-\pi_{i-2,b}})\\
        &-(z_{i-2,i+b-2}z_{i-1,\ell(\lambda)}q^{\pi_{i,b-1}-\pi_{i-2,b}}-z_{i-2,\ell(\lambda)}z_{i-1,i+b-1}q^{\pi_{ib}-\pi_{i-2,b+1}})-q^{g_{ib}^+(\sigma)-g_{ib}^+(\pi)}(1-z_{i-1,i+b-1}q^{\pi_{ib}})]\\&+z_{i-1,\ell(\lambda)}[(q^{\pi_{ib}-\pi_{i-1,b}}-q^{\pi_{i,b-1}-\pi_{i-1,b-1}})+ (q^{\pi_{i,b-1}}-q^{\pi_{ib}})+z_{i-2,\ell(\lambda)}[q^{\pi_{i-1,b}-\pi_{i-2,b+1}}-q^{\pi_{i-1,b+1}-\pi_{i-2,b+1}}]
    \end{split}\]
    for $b=j+1,\dots,j+\ell_{ij}-1$. Finally, define $E_{ij}:=\{(i,j),(i,j-1),(i-1,j),(i-1,j-1)\} \cup \{(i-1,j+\ell_{ij}),(i+\omega_{ij},j-1)\}$ for $(i,j) \in \mathcal{O}_1$  and we gather the remaining summands in terms of a function we denote $\Xi_{ij}^{(ab)}$ for $(a,b) \in E_{ij}$:
\begin{center}
\renewcommand{\arraystretch}{1.45}
\setlength{\tabcolsep}{0.6em}
\begin{tabular}{|>{\centering\arraybackslash}m{18em}|m{30em}|}
    \hline
    $(a,b) \in E_{ij}$ & $\Xi_{ij}^{(ab)}$ \\
    \hline
    $(i-1,j-1)$ & $z_{i-1,\ell(\lambda)}[z_{i-2,i+j-3}q^{\pi_{i,j-2}-\pi_{i-2,j-1}}+q^{\pi_{i,j-1}-\pi_{i-1,j-1}}+q^{\pi_{i,j-2}-\pi_{i-1,j-2}}-q^{\pi_{i,j-2}-\pi_{i-1,j-1}}-z_{i-2,i+j-3}q^{\pi_{i-1,j-1}-\pi_{i-1,j-2}+\pi_{i,j-2}-\pi_{i-2,j-1}}]-[b'_{i-1,j-1}(\pi,z;q)-\widetilde{b}_{i-1,j-1}(\pi,z;q)]+z_{i-2,\ell(\lambda)}[q^{\pi_{i-1,j-1}-\pi_{i-2,j}}-q^{\pi_{i-1,j}-\pi_{i-2,j}}]$ \\
    \hline
    $(i-1,j)$ & $-[b_{i-1,j}'(\pi,z;q)-\widetilde{b}_{i-1,j}(\pi,z;q)]+z_{i-2,\ell(\lambda)}[q^{\pi_{i-1,j}-\pi_{i-2,j+1}}-q^{\pi_{i-1,j+1}-\pi_{i-2,j+1}}]$\\
    \hline
    $(i,j-1)$ & $z_{i,\ell(\lambda)}[z_{i-1,i+j-2}q^{\pi_{i+1,j-2}-\pi_{i-1,j-1}}+q^{\pi_{i+1,j-1}-\pi_{i,j-1}}+q^{\pi_{i+1,j-2}-\pi_{i,j-2}}-q^{\pi_{i+1,j-2}-\pi_{i,j-1}}
        -z_{i-1,i+j-2}q^{\pi_{i,j-1}-\pi_{i,j-2}+\pi_{i+1,j-2}-\pi_{i-1,j-1}}]-[b'_{i,j-1}(\pi,z;q)-\widetilde{b}_{i,j-1}(\pi,z;q)]$\\
    \hline
    $(i,j)$ & $[\widetilde{b}_{ij}(\pi,z;q)-\widetilde{b}_{ij}(\sigma,z;q)]$ \\
    \hline
    $(i-1,j+\ell_{ij})$ & $-[b'_{i-1,j+\ell_{ij}}(\pi,z;q)-\widetilde{b}_{i-1,j+\ell_{ij}}(\pi,z;q)]+[\widetilde{b}_{i,j+\ell_{ij}}(\pi,z;q)-\widetilde{b}_{i,j+\ell_{ij}}(\sigma,z;q)]+z_{i-2,\ell(\lambda)}[q^{\pi_{i-1,j+\ell_{ij}}-\pi_{i-2,j+\ell_{ij}+1}}-q^{\pi_{i-1,j+\ell_{ij}+1}-\pi_{i-2,j+\ell_{ij}+1}}]$\\
    \hline
    $(i+\omega_{ij},j-1)$ & $[\widetilde{b}_{i+\omega_{ij},j-1}(\pi,z;q)-\widetilde{b}_{i+\omega_{ij},j-1}(\sigma,z;q)] -[b'_{i+\omega_{ij},j-1}(\pi,z;q)-\widetilde{b}_{i+\omega_{ij},j-1}(\sigma,z;q)]+z_{i+\omega_{ij},\ell(\lambda)}[z_{i+\omega_{ij}-1,i+\omega_{ij}+j-2}q^{\pi_{i+\omega_{ij}+1,j-2}-\pi_{i+\omega_{ij}-1,j-1}}+q^{\pi_{i+\omega_{ij}+1,j-1}-\pi_{i+\omega_{ij},j-1}}+q^{\pi_{i+\omega_{ij}+1,j-2}-\pi_{i+\omega_{ij},j-2}}-q^{\pi_{i+\omega_{ij}+1,j-2}-\pi_{i+\omega_{ij},j-1}}-z_{i+\omega_{ij}-1,i+\omega_{ij}+j-2}q^{\pi_{i+\omega_{ij},j-1}-\pi_{i+\omega_{ij},j-2}+\pi_{i+\omega_{ij}+1,j-2}-\pi_{i+\omega_{ij}-1,j-1}}]$\\
    \hline
\end{tabular}
\end{center}

For $(i,j) \in \mathcal{O}_1$ define
\[
\mathcal{B}_{ij}
=
-\sum_{a=i+1}^{i+\omega_{ij}-1} \Phi_{ij}^{(a)}(\pi)
-\sum_{b=j+1}^{j+\ell_{ij}-1} \Psi_{ij}^{(b)}(\pi)
-\sum_{(a,b) \in E_{ij}} \Xi_{ij}^{(ab)}(\pi)
+\sum_{a=i-1}^{i+\omega_{ij}} z_{a,\ell(\lambda)} .
\]
For $(i,j) \in \widetilde{C}(\mu)$ define
\begin{equation}\label{final1}
\begin{aligned}
\mathcal{B}_{ij}
={}&
-\Bigl[
    \widetilde{b}_{i+1,j}(\pi,z;q)
    -
    \widetilde{b}_{i+1,j}(\sigma,z;q)
\Bigr]
\mathbbm{1}_{(i+1,j)\in\widetilde{F}_{ij}}
-
\Bigl[
    \widetilde{b}_{i,j+1}(\pi,z;q)
    -
    \widetilde{b}_{i,j+1}(\sigma,z;q)
\Bigr]
\mathbbm{1}_{(i,j+1)\in\widetilde{F}_{ij}}
\\
&+
\Bigl[
    b_{ij}'(\pi,z;q)
    -
    b_{ij}(\pi,z;q)
\Bigr]
-
z_{i,\ell(\lambda)}
\Bigl[
    z_{i-1,i+j-1} q^{\pi_{i+1,j-1}-\pi_{i-1,j}}
    +
    q^{\pi_{i+1,j}-\pi_{i,j}}
    +
    q^{\pi_{i+1,j-1}-\pi_{i,j-1}}
\\
&
    -
    q^{\pi_{i+1,j-1}-\pi_{i,j}}
    -
    z_{i-1,i+j-1}
    q^{\pi_{i,j}-\pi_{i,j-1}+\pi_{i+1,j-1}-\pi_{i-1,j}}
\Bigr]
+
z_{i-1,\ell(\lambda)}
(1-q^{\pi_{i,j+1}})
\mathbbm{1}_{(i,j+1)\in\widetilde{F}_{ij}}
\\
&+
z_{i,\ell(\lambda)}
(1-q^{\pi_{i+1,j}})
\mathbbm{1}_{(i+1,j)\in\widetilde{F}_{ij}} .
\end{aligned}
\end{equation}

Observe that for $(i,j) \in \mathcal{O}_1$:
\begin{equation}\label{big2}
\begin{aligned}
\sum_{a=i+1}^{i+\omega_{ij}-1} \Phi_{ij}^{(a)}(\pi)
&=
z_{i,\ell(\lambda)}q^{-\pi_{ij}}
-
z_{i+\omega_{ij}-1,\ell(\lambda)}
q^{-\pi_{i+\omega_{ij}-1,j}}
+
\sum_{a=i+1}^{i+\omega_{ij}-1} z_{a,\ell(\lambda)}-
z_{i,\ell(\lambda)}
q^{\pi_{i+1,j-1}-\pi_{ij}}- 
\\
&
\sum_{a=i+1}^{i+\omega_{ij}-1}
z_{a-1,\ell(\lambda)}
(1-z_{a,a+j-1})
q^{-\pi_{a-1,j}}
(1-q^{\pi_{aj}})+z_{i+\omega_{ij}-1,\ell(\lambda)}
q^{\pi_{i+\omega_{ij},j-1}-\pi_{i+\omega_{ij}-1,j}},
\end{aligned}
\end{equation}

\begin{equation}\label{big1}
\begin{aligned}
&\sum_{b=j+1}^{j+\ell_{ij}-1} \Psi_{ij}^{(b)}(\pi)
={}
-
z_{i-2,\ell(\lambda)}
z_{i-1,i+j+\ell_{ij}-2}
q^{\pi_{i,j+\ell_{ij}-1}
+\pi_{i-1,j+\ell_{ij}}
-\pi_{i-1,j+\ell_{ij}-1}
-\pi_{i-2,j+\ell_{ij}}}-z_{i-1,\ell(\lambda)}q^{\pi_{i,j+\ell_{ij}-1}}
\\
&\qquad+
z_{i-2,i+j-1}
z_{i-1,\ell(\lambda)}
q^{\pi_{i-1,j+1}
-\pi_{i-1,j}
+\pi_{ij}
-\pi_{i-2,j+1}}
+
z_{i-2,\ell(\lambda)}
z_{i-1,i+j+\ell_{ij}-2}
q^{\pi_{i,j+\ell_{ij}-1}-\pi_{i-2,j+\ell_{ij}}}
\\
&\qquad\qquad -
z_{i-1,\ell(\lambda)}
z_{i-2,i+j-1}
q^{\pi_{ij}-\pi_{i-2,j+1}}
+
z_{i-1,\ell(\lambda)}
q^{\pi_{i,j+\ell_{ij}-1}-\pi_{i-1,j+\ell_{ij}-1}}
-
z_{i-1,\ell(\lambda)}
q^{\pi_{ij}-\pi_{i-1,j}}+z_{i-1,\ell(\lambda)}q^{\pi_{ij}}
 ,
\end{aligned}
\end{equation}

\[
\begin{aligned}
&\sum_{(a,b)\in\{(i,j),(i,j-1),(i-1,j),(i-1,j-1)\}}
\Xi_{ij}^{(ab)}(\pi)
=
z_{i,\ell(\lambda)}
+
z_{i-1,\ell(\lambda)}
-
z_{i-1,\ell(\lambda)}
(1-z_{i,i+j-1})
(1-q^{\pi_{ij}})
\\
&-
\Bigl(
    z_{i,\ell(\lambda)}q^{-\pi_{ij}}
    +
    z_{i,\ell(\lambda)}
    q^{\pi_{i+1,j-1}-\pi_{ij}}
    -
    z_{i-1,i+j-1}
    z_{i-2,\ell(\lambda)}
    q^{\pi_{i-1,j+1}-\pi_{i-1,j}+\pi_{ij}-\pi_{i-2,j+1}}+z_{i-1,\ell(\lambda)}q^{\pi_{ij}}
\\
&
    \qquad -
    z_{i-1,\ell(\lambda)}
    z_{i-2,i+j-1}
    q^{\pi_{ij}-\pi_{i-2,j+1}}
    -
    z_{i-1,\ell(\lambda)}
    q^{\pi_{ij}-\pi_{i-1,j}}
\Bigr).
\end{aligned}
\]

To understand the last two points in $E_{ij}$, i.e.
$(i-1,j+\ell_{ij})$ and $(i+\omega_{ij},j-1)$, it is necessary to understand the pairings between hook shapes that appear in $\hat{\mu}/\mu$:
\begin{enumerate}
    \item If $v=(i,j) \in \mathcal{O}_1^0$, then there are four possibilities concerning the points
    $x=(i-2,j+\ell_{ij}+1)$ and $y=(i+\omega_{ij}+1,j-2)$.
    The first one we deal with is $x,y \in \widetilde{C}(\mu)$. The core calculations to consider here are
    \[
    \begin{aligned}
    &\Bigl[
        \widetilde{b}_{i-1,j+\ell_{ij}+1}(\pi,z;q)
        -
        \widetilde{b}_{i-1,j+\ell_{ij}+1}(\sigma,z;q)
    \Bigr]
    +
    \Xi_{ij}^{(i-1,j+\ell_{ij})}(\pi)
    =
    z_{i-1,\ell(\lambda)}
    \Bigl[
        z_{i-2,i+j+\ell_{ij}-1}q^{\pi_{i,j+\ell_{ij}}}
        \\&\qquad-
        z_{i-2,i+j+\ell_{ij}-1}
        q^{\pi_{i,j+\ell_{ij}}+\pi_{i-1,j+\ell_{ij}+1}}
    \Bigr]-
    \Bigl(
        -
        z_{i-2,\ell(\lambda)}
        z_{i-1,i+j+\ell_{ij}-2}
        q^{\pi_{i,j+\ell_{ij}-1}
        +\pi_{i-1,j+\ell_{ij}}
        -\pi_{i-1,j+\ell_{ij}-1}
        -\pi_{i-2,j+\ell_{ij}}}
    \\&\qquad\qquad
        +
        z_{i-2,\ell(\lambda)}
        z_{i-1,i+j+\ell_{ij}-2}
        q^{\pi_{i,j+\ell_{ij}-1}-\pi_{i-2,j+\ell_{ij}}}
        +
        z_{i-1,\ell(\lambda)}
        q^{\pi_{i,j+\ell_{ij}-1}-\pi_{i-1,j+\ell_{ij}-1}}
        -
        z_{i-1,\ell(\lambda)}
        q^{\pi_{i,j+\ell_{ij}-1}}
    \Bigr).
    \end{aligned}
    \]

    and
    \begin{equation}\label{final3}
    \begin{aligned}
    &-\mathcal{B}_y
    +
    \Xi_{ij}^{(i+\omega_{ij},j-1)}(\pi)
    +
    z_{i+\omega_{ij},\ell(\lambda)}
    (1-q^{\pi_{i+\omega_{ij}+1,j-1}})
    =
    -
    z_{i+\omega_{ij}+1,\ell(\lambda)}
    z_{i+\omega_{ij},i+\omega_{ij}+j-1}
    q^{\pi_{i+\omega_{ij}+2,j-2}
    +\pi_{i+\omega_{ij}+1,j-1}}
    \\
    &\qquad +
    z_{i+\omega_{ij}+1,\ell(\lambda)}
    z_{i+\omega_{ij},i+\omega_{ij}+j-1}
    q^{\pi_{i+\omega_{ij}+2,j-2}}+
    z_{i+\omega_{ij}+1,\ell(\lambda)}
    q^{\pi_{i+\omega_{ij}+1,j-1}}+
    z_{i+\omega_{ij},\ell(\lambda)}
    \Bigl[
        q^{\pi_{i+\omega_{ij}+1,j-1}}
    \\
    &\qquad\qquad
        +
        z_{i+\omega_{ij}-1,i+\omega_{ij}+j-1}
        q^{\pi_{i+\omega_{ij}+1,j-1}
        -\pi_{i+\omega_{ij}-1,j}}
                -
        z_{i+\omega_{ij}-1,i+\omega_{ij}+j-1}
        q^{\pi_{i+\omega_{ij}+1,j-1}
        +\pi_{i+\omega_{ij},j}
        -\pi_{i+\omega_{ij}-1,j}}
    \Bigr].
    \end{aligned}
    \end{equation}

    Thus we have shown that
    \begin{equation}\label{final2}
    \begin{aligned}
    &\Bigl[
        \widetilde{b}_{i-1,j+\ell_{ij}+1}(\pi,z;q)
        -
        \widetilde{b}_{i-2,j+\ell_{ij}+1}(\sigma,z;q)
    \Bigr]
    -
    \mathcal{B}_{ij}
    -
    \mathcal{B}_y
    -
    z_{i-2,\ell(\lambda)}
    (1-q^{\pi_{i-1,j+\ell_{ij}+1}})=\\
    &-
    z_{i-1,\ell(\lambda)}
    (1-z_{i,i+j-1})
    (1-q^{\pi_{ij}})-
    z_{i-1,\ell(\lambda)}
    z_{i-2,i+j+\ell_{ij}-1}
    (1-q^{\pi_{i-1,j+\ell_{ij}+1}})
    (1-q^{\pi_{i,j+\ell_{ij}}})
    \\
    &-
    z_{i+\omega_{ij}+1,\ell(\lambda)}
    z_{i+\omega_{ij},i+\omega_{ij}+j-1}
    (1-q^{\pi_{i+\omega_{ij}+1,j-1}})
    (1-q^{\pi_{i+\omega_{ij}+2,j-2}})
    \\
    &-
    z_{i+\omega_{ij},\ell(\lambda)}
    z_{i+\omega_{ij}-1,i+\omega_{ij}+j-1}
    q^{-\pi_{i+\omega_{ij}-1,j}}
    (1-q^{\pi_{i+\omega_{ij}+1,j-1}})
    (1-q^{\pi_{i+\omega_{ij},j}})
    \\
    &-
    \sum_{a=i+1}^{i+\omega_{ij}}
    z_{a-1,\ell(\lambda)}
    (1-z_{a,a+j-1})
    q^{-\pi_{a-1,j}}
    (1-q^{\pi_{aj}})
    \\
    &-
    z_{i-2,\ell(\lambda)}
    (1-z_{i-1,i+j+\ell_{ij}-1})
    (1-q^{\pi_{i-1,j+\ell_{ij}+1}})
    \\
    &-
    z_{i+\omega_{ij},\ell(\lambda)}
    (1-z_{i+\omega_{ij}+1,i+\omega_{ij}+j-1})
    (1-q^{\pi_{i+\omega_{ij}+1,j-1}}).
    \end{aligned}
    \end{equation}

    Now suppose instead that $x\in \lambda/\mu$ and
    $(i-1,j+\ell_{ij})\in G_{i',j'}$ for some
    $(i',j') \in \mathcal{O}_1^- \cup \mathcal{O}_1^{\bullet}$ and
    $y \in \widetilde{C}(\mu)$.
    We denote $u=(i-1,j+\ell_{ij}), v=(i,j+\ell_{ij}),w=(i-1,j+\ell_{ij}+1).$ One can see that
    \[
    \begin{aligned}
    &\Xi_{i',j'}^{u}
    +
    \Bigl[
        \widetilde{b}_v(\pi,z;q)
        -
        \widetilde{b}_v(\sigma,z;q)
    \Bigr]
    +
    z_{i-2,\ell(\lambda)}
    \Bigl[
        q^{\pi_{i-1,j+\ell_{ij}}-\pi_{i-2,j+\ell_{ij}+1}}
        -
        q^{\pi_{i-1,j+\ell_{ij}+1}-\pi_{i-2,j+\ell_{ij}+1}}
    \Bigr]
    \\
    &\quad
    +
    z_{i-1,\ell(\lambda)}
    \Bigl(
        q^{\pi_{i,j+\ell_{ij}-1}-\pi_{i-1,j+\ell_{ij}}}
        -
        q^{\pi_{i,j+\ell_{ij}-1}}
    \Bigr)
    =
    \end{aligned}
    \]
    \begin{equation}\label{final4}
    \begin{aligned}
    &\Xi_{ij}^u
    +
    \Bigl[
        \widetilde{b}_w(\pi,z;q)
        -
        \widetilde{b}_w(\sigma,z;q)
    \Bigr]
+
    z_{i-1,\ell(\lambda)}
    \Bigl[
        z_{i-2,i+j+\ell_{ij}-2}
        q^{\pi_{i,j+\ell_{ij}-1}-\pi_{i-2,j+\ell_{ij}}}
        +
        q^{\pi_{i,j+\ell_{ij}-1}-\pi_{i-1,j+\ell_{ij}-1}}
    \\
    &\qquad
        -
        q^{\pi_{i,j+\ell_{ij}-1}}
        -
        z_{i-2,i+j+\ell_{ij}-2}
        q^{\pi_{i-1,j+\ell_{ij}}
        -\pi_{i-1,j+\ell_{ij}-1}
        +\pi_{i,j+\ell_{ij}-1}
        -\pi_{i-2,j+\ell_{ij}}}
    \Bigr]
=
    z_{i-2,\ell(\lambda)}
    \Bigl[
        q^{\pi_{i-1,j+\ell_{ij}+1}-\pi_{i-2,j+\ell_{ij}+1}}
    \\
    &\qquad\qquad
        -
        z_{i-1,i+j+\ell_{ij}-1}
        q^{\pi_{i,j+\ell_{ij}}
        +\pi_{i-1,j+\ell_{ij}+1}
        -\pi_{i-2,j+\ell_{ij}+1}}
        +
        z_{i-1,i+j+\ell_{ij}-1}
        q^{\pi_{i,j+\ell_{ij}}-\pi_{i-2,j+\ell_{ij}+1}}
        -
        q^{-\pi_{i-2,j+\ell_{ij}+1}}
    \Bigr]=
    \end{aligned}
    \end{equation}
    \[    -
    z_{i-2,\ell(\lambda)}
    q^{-\pi_{i-2,j+\ell_{ij}+1}}[z_{i-1,i+j+\ell_{ij}-1}
    (1-q^{\pi_{i,j+\ell_{ij}}})
    (1-q^{\pi_{i-1,j+\ell_{ij}+1}})
    -
    (1-z_{i-1,i+j+\ell_{ij}-1})
    (1-q^{\pi_{i-1,j+\ell_{ij}+1}})]\]
    where this latter expression is the contribution coming from the interactions between the point
    $u \in \mu$ and its nearest-neighbors. The remaining case where
    $y \in \lambda/\mu$ and $(i+\omega_{ij},j-1) \in G_{i',j'}$ for some
    $(i',j') \in \mathcal{O}_1^+ \cup \mathcal{O}_1^{\bullet}$ can be handled in the exact same manner.

    \item If $(i,j) \in \widetilde{C}(\mu)$, then it must hold that
    $(i-1,j+1),(i+1,j-1) \in \mu$ with
    $(i-1,j+2),(i+1,j) \in \lambda/\mu$.
    The first case we are interested in is $\widetilde{F}_{ij}= \{(i+1,j),(i,j+1)\},\widetilde{F}_{i+1,j-1}=\emptyset$, in which case we denote
    \[
        y=(i,j),\qquad
        v=(i+1,j),\qquad
        u=(i+1,j-1),\qquad
        w=(i+2,j-1),
    \]
    and observe that
    \begin{equation}\label{final7}
    \begin{aligned}
    &\Bigl[
        \widetilde{b}_v(\pi,z;q)
        -
        \widetilde{b}_v(\sigma,z;q)
    \Bigr]
    -
    \mathcal{B}_u
    -
    z_{i,\ell(\lambda)}
    (1-q^{\pi_v})
    -
    z_{i+1,\ell(\lambda)}
    q^{\pi_w}
    \\
    ={}&
    -
    z_{i+1,\ell(\lambda)}
    z_{i,i+j}
    (1-q^{\pi_w})
    (1-q^{\pi_v})
    -
    z_{i,\ell(\lambda)}
    (1-z_{i+1,i+j})
    (1-q^{\pi_v}) .
    \end{aligned}
    \end{equation}

    The second case we investigate is
    $\widetilde{F}_{ij}=\{(i+1,j)\}$, in which case we denote
    \[
        x=(i,j),\qquad
        v=(i,j+1),\qquad
        u=(i+1,j),
    \]
    and observe that
    \begin{equation}\label{final8}
    \begin{aligned}
    &\Bigl[
        \widetilde{b}_v(\pi,z;q)
        -
        \widetilde{b}_v(\sigma,z;q)
    \Bigr]
    -
    \mathcal{B}_x
    -
    \Bigl[
        \widetilde{b}_u(\pi,z;q)
        -
        \widetilde{b}_u(\sigma,z;q)
    \Bigr]+
    z_{i,\ell(\lambda)}
    (1-q^{\pi_{i+1,j}})    -
    z_{i,\ell(\lambda)}
    q^{\pi_{i+1,j}}
    \\
    &\quad
    -
    z_{i-1,\ell(\lambda)}
    (1-q^{\pi_{i,j+1}})
=
    -
    z_{i,\ell(\lambda)}
    z_{i-1,i+j}
    (1-q^{\pi_{i+1,j}})
    (1-q^{\pi_{i,j+1}})-
    z_{i-1,\ell(\lambda)}
    (1-z_{i,i+j})
    (1-q^{\pi_{i,j+1}}).
    \end{aligned}
    \end{equation}

    For the remaining case $\widetilde{F}_{ij}=\{(i+1,j),(i,j+1)\},\widetilde{F}_{i+1,j-1}=\{(i+2,j-1)\}$ the contribution $-\mathcal{B}_{ij}$ to the sum directly from the point $(i,j)$ is $-
    z_{i-1,\ell(\lambda)}
    (1-q^{\pi_{i,j+1}})
    -
    z_{i,\ell(\lambda)}
    q^{\pi_{i+1,j}}.$ We briefly note the telescoping that occurs for adjacent elements of
    $\widetilde{C}(\mu)$. Suppose
    $(i,j),(a,b) \in \widetilde{C}(\mu)$ for which
    $(i-1,j+1),(a+1,b-1) \in \bigcup_{u \in \mathcal{O}_1}G_u$
    and $a-i=b-j$ with $i \leq a$, $j \geq b$. Then we may write $-
    \sum_{s=i}^{a-1}
    z_{s,\ell(\lambda)}
    \bigl(1-q^{\pi_{s+1,j-(s-i)}}\bigr)
    -
    \sum_{s=i-1}^{a-1}
    z_{s+1,\ell(\lambda)}
    q^{\pi_{s+2,j-(s-i)-1}}=
    -
    \sum_{s=i-1}^{a-2}
    z_{s+1,\ell(\lambda)}
    -
    z_{a,\ell(\lambda)}
    q^{\pi_{a+1,b}}$ where the $z_{a,\ell(\lambda)}q^{\pi_{a+1,b}}$ term will eventually cancel with terms from
    \eqref{final2}. The leftover $z$ terms are a result of over-counting we have allowed ourselves in the course of this proof.

    \item If
    $v=(i,j) \in \mathcal{O}_1^+ \cup \mathcal{O}_1^- \cup \mathcal{O}_1^{\bullet}$,
    we can once again use \eqref{final4} to understand the edge which is formed by the intersection of two
    $G_{ij},G_{i',j'}$ hooks.
\end{enumerate}

As we stated in the beginning of the proof, our calculations were done with nondegenerate hooks in mind.
The above calculations hold for degenerate hooks with the only difference in the calculations being slight differences in the way hooks are indexed.
For instance, if $F_{ij}$ is a degenerate hook of vertical length $1$, our conventions allow
$(i-1,j+\ell_{ij})\in \lambda/\mu$, and this was not true for nondegenerate hooks.
This means that $(i-1,j+\ell_{ij}-1)\in \mu$ is the element of $G_{ij}$ with largest vertical coordinate rather than $(i-1,j+\ell_{ij})$, and the calculations we did above would have to be adjusted to reflect that.\newline
\indent Now, one may take the sum $\sum_{(i,j) \in \widetilde{C}(\mu) \cup \mathcal{O}_1} \mathcal{B}_{ij}$
and discard terms that were over-counted, noting that each element was counted at most twice.
In particular, intersections between staircases
$\mathscr{S}(\hat{G}_{ij})$ and $\mathscr{S}(\hat{G}_{i',j'})$ form another staircase shape, and subtracting away the extra terms from this intersection implies that we can neglect contributions coming from these intersections.
Another example of over-counting is the last two terms of \eqref{final1}.
Although these terms make the cancellations in \eqref{final2}, \eqref{final7}, and \eqref{final8} nicer, they are once again a result of the over-counting that occurs when we perform the summation \eqref{generalstaircasestuff} for each individual staircase shape.
Thus, $T_{\lambda,\mu} + U_{\lambda,\mu} = V_{\lambda,\mu}$ follows from \eqref{final5}.

    \end{proof}
    \section{The $q \to 1$ Limit}
    \indent From \eqref{ok9} it can be shown that
    \[
            \tilde{a}_r(n):=\limmy{q}{1}{(1-q)^{2n}a_r(n;q)} = \sum_{\pi \in \Pi^r_n} \prod_{1 \leq i+j-1 \leq r} \frac{1}{(\pi_{ij}-\pi_{i-1,j})!(\pi_{ij}-\pi_{i,j-1})!},
    \]
    where the right-hand-side were the coefficients considered in \cite{neil}. From this fact we can immediately deduce the "scaled" version of $h^r_n$ and show that its $q \to 1$ limit coincides with the corresponding operator from \cite{neil}.\newline
    \indent The sum $\sum_{i=0}^r (q^{n_{i+1}}-q^{n_i})=0$ is telescoping and can be added to $h^r_n$ in order to get the form
        \[
            h^r_n=\sum_{i=0}^r \bigg[q^{n_{i+1}-n_i}L_{n_i}+(1-q^{n_{i+1}-n_i})(1-q^{n_i})\bigg].
        \]
        Now consider the conjugated operator
        \[
            \tilde{h}^r_n=(1-q)^{2n} \circ h^r_n \circ (1-q)^{-2n}=\sum_{i=0}^r \bigg[q^{n_{i+1}-n_i}(1-q)^2L_{n_i}+(1-q^{n_{i+1}-n_i})(1-q^{n_i})\bigg].
        \]
        By dividing the coefficients of $\tilde{h}^r_n$ by $(1-q)^2$, it becomes apparent that $(1-q)^{-2} \tilde{h}^r_n$ converges weakly to $\displaystyle \sum_{i=0}^r [L_{n_i}+n_i(n_{i+1}-n_i)]$. The latter operator is referred to as $h^r_n$ in \cite{neil}. Lastly, recall from \eqref{final6} that $V_{\lambda,\mu}(\sigma,z;q)$ was the potential for which $(G^{\lambda/\mu,\alpha}+V_{\lambda,\mu})A_{\lambda,\mu}=0$ held in the setting of Subsection $5.2$. We thus have
    \[\limmy{q}{1}{(1-q)^{-2}V_{\lambda,\mu}(\sigma,z;q)}=\sum_{(i,j) \in C(\mu)} \sigma_{i+1,j}\sigma_{i,j+1} + \sum_{i=1}^{\ell(\mu)} \beta_{i+1,\mu_i}\sigma_{i,\mu_i+1},\]
    where the latter function is the analogous potential denoted $V_{\lambda,\mu}(\sigma)$ in \cite{neil}.
\end{appendices}

\end{document}